\documentclass[12pt]{scrartcl}
\usepackage[english]{babel}
\usepackage{amsmath}
\usepackage{amsthm}
\usepackage{txfonts}
\usepackage{algorithmic}
\usepackage{tikz}
\usepackage{mathtools}
\usepackage{subcaption}
\usepackage{slashbox}
\usetikzlibrary{calc,decorations.pathreplacing}

\theoremstyle{definition}
\newtheorem{alg}{Algorithm}[section]
\newtheorem{df}[alg]{Definition}
\newtheorem{ex}[alg]{Example}
\theoremstyle{remark}
\newtheorem*{rem}{Remark}
\theoremstyle{plain}
\newtheorem{lma}[alg]{Lemma}
\newtheorem{thm}[alg]{Theorem}
\newtheorem{prp}[alg]{Proposition}
\newtheorem{crl}[alg]{Corollary}

\DeclareMathOperator \xx {\pmb{\mathrm x}}
\DeclareMathOperator \yy {\pmb{\mathrm y}}
\DeclareMathOperator \XX {\pmb{\mathrm X}}
\DeclareMathOperator \YY {\pmb{\mathrm Y}}
\DeclareMathOperator*\tcup {{\textstyle\bigcup}}
\DeclareMathOperator \minset {Min_\subseteq}
\DeclareMathOperator \Dist {Dist}
\DeclareMathOperator \bisect {bisect}
\DeclareMathOperator \child {child}
\DeclareMathOperator \midp {mid}
\DeclareMathOperator \size {size}
\DeclareMathOperator \ptb {Ptb}

\newcommand \mathbox [2][2em]{\makebox[#1]{$\displaystyle #2$}}
\newlength \relwidth 
\newcommand \Stackrel [2]{\mathrel{\mathbox[\relwidth]{\stackrel{#1}{#2}}}}
\newcommand \mbigcup [1][]{\mathbox{\bigcup_{#1}}}
\newlength \Omegaheight 
\newcommand \barOmega {\overline{\rule{0pt}{.875\Omegaheight}\smash\Omega}}
\newcommand \sei \coloneqq
\newcommand \N {\mathcal N}
\newcommand \AR {\mathcal{AR}}
\newcommand \hsk {\mathrm{hSk}}
\newcommand \vsk {\mathrm{vSk}}
\newcommand \sk {\mathrm{Sk}}
\newcommand \raiseqed {\\[-3.25em]}

\newcommand \pq [3][p,q]{#2^{#1}(#3)}
\newcommand \clos [1][p,q]{\operatorname{clos}^{#1}}
\newcommand \D [1][p,q]{\operatorname{\mathbf D}^{#1}}
\newcommand \U [2][p,q]{\overline U^{#1}(#2)}
\newcommand \ho {_{\mathrm{ho}}}
\newcommand \refine [1][p,q]{\operatorname{ref}^{#1}}
\newcommand \M {\mathcal M}
\newcommand \KM {\tilde K}
\newcommand \G {\mathcal G}
\newcommand \Guni [1]{\mathcal G_{u\mid #1}}
\newcommand \A [1][p,q]{\mathbb A^{#1}}
\newcommand \ext [1][p,q]{\operatorname{ext}^{#1}}
\newcommand \ceilfrac [2] {\bigl\lceil\tfrac{#1}{#2}\bigr\rceil}
\newcommand \floorfrac [2] {\bigl\lfloor\tfrac{#1}{#2}\bigr\rfloor}
\newcommand \sqr[1]{\negmedspace\sqrt{#1}}
\newcommand \Mtilde {\overset{\mspace{9mu}\textstyle\sim}{\smash{\M}\rule{0pt}{.8ex}}}

\newcounter{num}
\numberwithin{num}{alg}
\renewcommand{\thenum}{\arabic{num}}
\newcommand{\pnumpx}[1][]{\refstepcounter{num}\noindent\textbf{(\thenum)}\enspace\emph{#1}}

\newcommand{\numref}[1]{\textbf{(\ref{#1})}}

\title{Analysis-suitable adaptive T-mesh refinement with linear complexity}
\author{Philipp Morgenstern\footnote{Institute for Numerical Simulation, Rheinische Friedrich-Wilhelms-Universit\"at Bonn -- \texttt{morgenstern@ins.uni.bonn.de}}\enspace and Daniel Peterseim\footnote{Institute for Numerical Simulation, Rheinische Friedrich-Wilhelms-Universit\"at Bonn -- \texttt{peterseim@ins.uni.bonn.de}}}

\begin{document}
\maketitle
\begin{abstract}
We present an efficient adaptive refinement procedure that preserves analysis-suitability of the T-mesh, this is, the linear independence of the T-spline blending functions. We prove analysis-suitability of the overlays and boundedness of their cardinalities, nestedness of the generated T-spline spaces, and linear computational complexity of the refinement procedure in terms of the number of  marked and generated mesh elements. 
\end{abstract}


\section{Introduction}
T-splines \cite{Tsplines-birth} have been introduced as a free-form geometric technology
and are one of the most promising features in the Isogeometric Analysis (IGA) framework introduced by Hughes, Cottrell and Basilevs \cite{IGA-birth,IGA-longbirth}.
At present, the main interest in IGA is in finding discrete function spaces that integrate well into CAD applications and, at the same time, can be used for  Finite Element Analysis. 
Throughout the last years, hierarchical B-Splines \cite{Spline-forests, IGA-with-hier-Splines} and LR-Splines \cite{LR-Splines, IGA-with-LR-Splines} have arisen as alternative approaches to T-Splines  for the establishment of an adaptive B-Spline technology. While none of these strategies has outperformed the other competing approaches until today, this paper aims to push forward and motivate the T-Spline technology.

Since T-splines can be locally refined \cite{Tspline-refinement1}, they potentially link the powerful geometric concept of Non-Uniform Rational B-Splines (NURBS) to meshes with T-junctions (referred as ``hanging nodes'' in the Finite Element context) and, hence, the well-established framework of adaptive mesh refinement. 
However, in \cite{particular-Tmeshes},  it was shown that T-meshes can induce linear dependent T-spline blending functions. This prohibits the use of T-splines as a basis for analytical purposes such as solving a partial differential equation. In particular, the mesh refinement algorithm presented in \cite{Tspline-refinement1} does not preserve analysis-suitability in general. 
This insight motivated the research on T-meshes that guarantee the linear independence of the corresponding T-spline blending functions, referred to as \emph{analysis-suitable T-meshes}.
Analysis-suitability has been characterized in terms of topological mesh properties in $2d$ \cite{lin-ind-Tsplines} and, in an alternative approach, through the equivalent concept of Dual-Compatibility \cite{AS-Tsplines-are-DC}, which allows for generalization to three-dimensional meshes.

A refinement procedure that preserves the analysis-suitability of two-dimensional T-meshes was finally presented in \cite{Tspline-refinement2}. 
The procedure first refines the marked elements, producing a mesh that is not analysis-suitable in general, and then  computes a refinement which is analysis-suitable and generates a T-spline space that is a superspace of the previous one. 
This second refinement involves heuristic local estimates on how much refinement is needed to achieve the desired properties. 
Hence, the reliable theoretical analysis of the algorithm is very difficult and so is the analysis of corresponding automatic mesh refinement algorithms driven by a posteriori error estimators. Such analysis is currently available only for triangular meshes \cite{Axioms-of-Adaptivity,CKNS,Stevenson}, but is necessary to reliably point out the advantages of adaptive mesh refinement.

In this paper, we present a new refinement algorithm which provides 
\begin{enumerate}
\item the preservation of analysis-suitability and nestedness of the generated T-spline spaces,
\item a bounded cardinality of the overlay (which is the coarsest common refinement of two meshes),
\item linear computational complexity of the refinement procedure in the sense that there is a constant bound, depending only on the polynomial  degree of the T-spline blending functions, on the ratio between the number of generated elements in the fine mesh and the number of marked elements in all refinement steps.
\end{enumerate}

This paper is organized as follows. We define the refinement algorithm along with a class of admissible meshes in Section~\ref{sec: refinement}. In Section~\ref{sec: AS}, we prove that all admissible meshes are analysis-suitable. Section~\ref{sec: overlay} proves essential properties of the overlay of two admissible meshes, and in Section~\ref{sec: nestedness} we prove nestedness of the T-spline spaces corresponding to admissible refinements. 
Section~\ref{sec: complexity} shows linear complexity of the refinement procedure, and conclusions and an outlook to future work are finally given in Section~\ref{sec: conclusions}.
The Sections \ref{sec: AS}, \ref{sec: overlay} and \ref{sec: complexity} independently rely on the definitions and results of Section~\ref{sec: refinement}, Section~\ref{sec: nestedness} also makes use of the definitions from Section~\ref{sec: overlay}.

\section{Adaptive mesh refinement}\label{sec: refinement}
This section defines the new refinement algorithm and characterizes the class of meshes which is generated by this algorithm. The initial mesh is assumed to have a very simple structure. In the context of IGA, the partitioned rectangular domain is referred to as \emph{index domain}. This is, we assume that the \emph{physical domain} (on which, e.g., a PDE is to be solved) is obtained by a continuous map from the active region (cf. Section~\ref{sec: AS}), which is a subset of the index domain. Throughout this paper, we focus on the mesh refinement only, and therefore we will only consider the index domain. For the parametrization and refinement of the T-spline blending functions, we refer to \cite{Tspline-refinement2}.
\begin{df}[Initial mesh, element]
Given positive numbers $M,N\in\mathbb N$, the initial mesh $\G_0$ is a tensor product mesh consisting of closed squares (also denoted \emph{elements}) with side length 1, i.e.,
\[\G_0\sei\Bigl\{[m-1,m]\times[n-1,n]\mid m\in\{1,\dots,M\},n\in\{1,\dots,N\}\Bigr\}.\]
The domain partitioned by $\G_0$ is denoted by $\barOmega\sei\tcup\G_0$.
\end{df}

The key property of the refinement algorithm will be that refinement of an element $K$ is allowed only if elements in a certain neighbourhood are sufficiently fine. The size of this neighbourhood, which is denoted \emph{$(p,q)$-patch} and defined through the definitions below, depends on the size of $K$ and the polynomial bi-degree $(p,q)$ of the T-spline blending functions. 

\begin{df}[Level]
The \emph{level} of an element $K$ is defined by \[\ell(K)\sei-\log_2|K|,\] where $|K|$ denotes the volume of $K$.
This implies that all elements of the initial mesh have level zero and that the bisection of an element $K$ yields two elements of level $\ell(K)+1$.
\end{df}

\begin{df}[Vector-valued distance]\label{df: distance}
Given $x\in\barOmega$ and an element $K$, we define their distance 
as the componentwise absolute value of the  difference between $x$ and the midpoint of $K$,
\[\Dist(K,x)\sei\operatorname{abs}\bigl(\midp(K)-x\bigr)\ \in\mathbb R^2.\]
For two elements $K_1,K_2$, we define the shorthand notation 
\[\Dist(K_1,K_2)\sei\operatorname{abs}\bigl(\midp(K_1)-\midp(K_2)\bigr).\]
\end{df}

\begin{df}\label{df: magic patch}
Given an element $K$ and polynomial degrees $p$ and $q$, the $(p,q)$-patch is defined by 
\[\pq\G K \sei \bigl\{K'\in\G\mid\Dist(K',K)\le\D(\ell(K))\bigr\},\]
where \[\D(k)=\begin{cases}
2^{-k/2}\left(\floorfrac p2+\tfrac12,\,\ceilfrac q2+\tfrac12\right)&\text{if $k$ is even,}\\
2^{-(k+1)/2}\left(\ceilfrac p2+\tfrac12,\,2\floorfrac q2+1\right)&\text{if $k$ is odd.}
\end{cases}\]
Note as a technical detail that this definition does \emph{not} require that $K\in\G$.
\end{df}
\begin{rem}
In a uniform even-leveled mesh, $\tcup\pq\G K$ is  obtained by extending $K$ by a face extension length (cf. Definition~\ref{df: TJ-extensions}) above and below and by an edge extension length to the left and to the right. 
In a uniform odd-leveled mesh, $\tcup\pq\G K$ is obtained by extending $K$ by a face extension length to the left and to the right and by an edge extension length above and below. 
The $(p,q)$-patch will be used to enforce a local quasi-uniformity of the mesh.
Throughout the rest of this paper, we assume $p,q\ge2$. This guarantees that neighboring elements of $K$ (elements that share an edge or vertex with $K$) are always in $\pq\G K$, and that nested elements $\check K\subseteq \hat K$ have nested $(p,q)$-patches $\pq\G {\check K}\subseteq\pq\G {\hat K}$.
\end{rem}

\begin{figure}[ht]
\centering
\includegraphics[width=.3\textwidth]{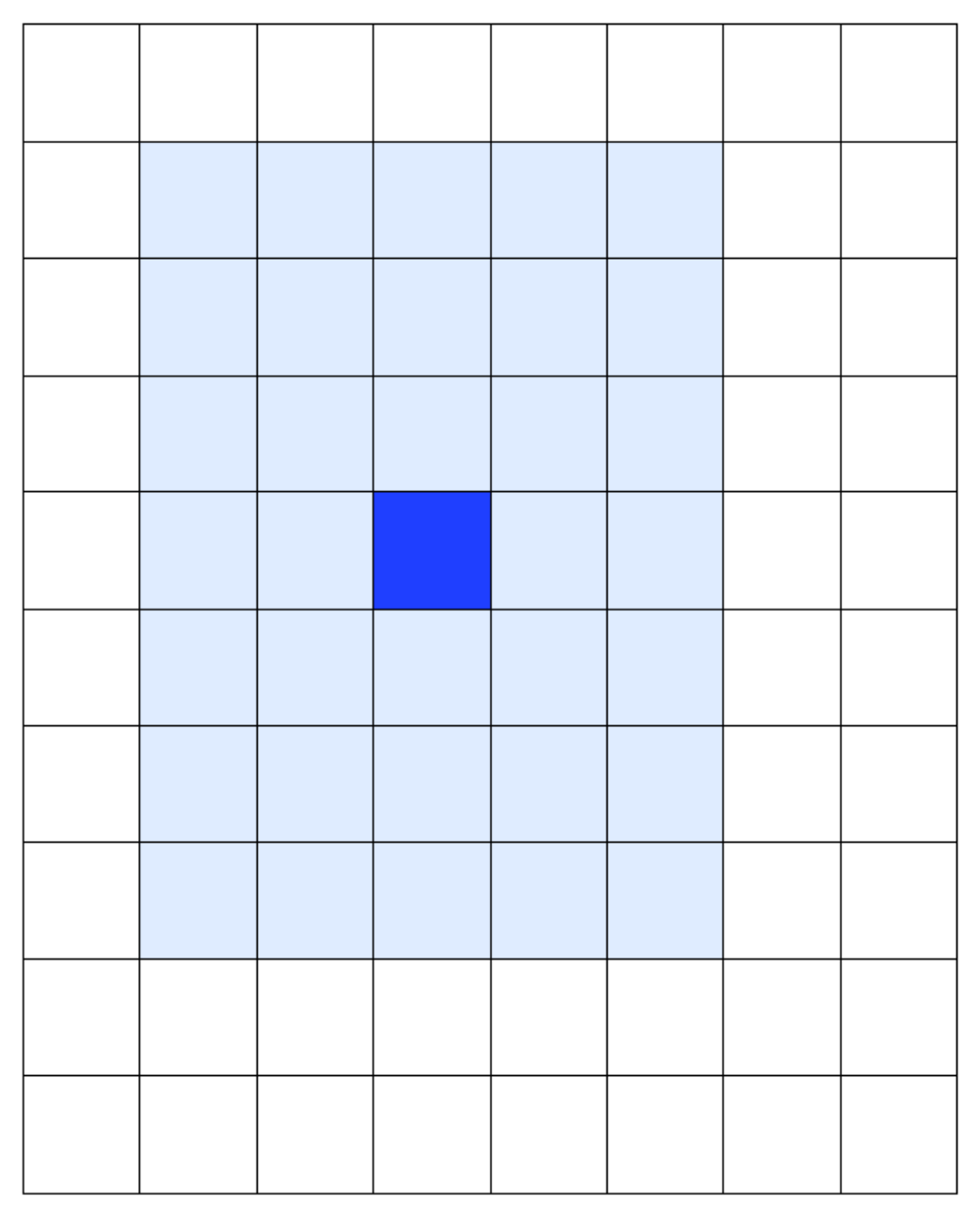}
\hspace{.12\textwidth}
\includegraphics[width=.3\textwidth]{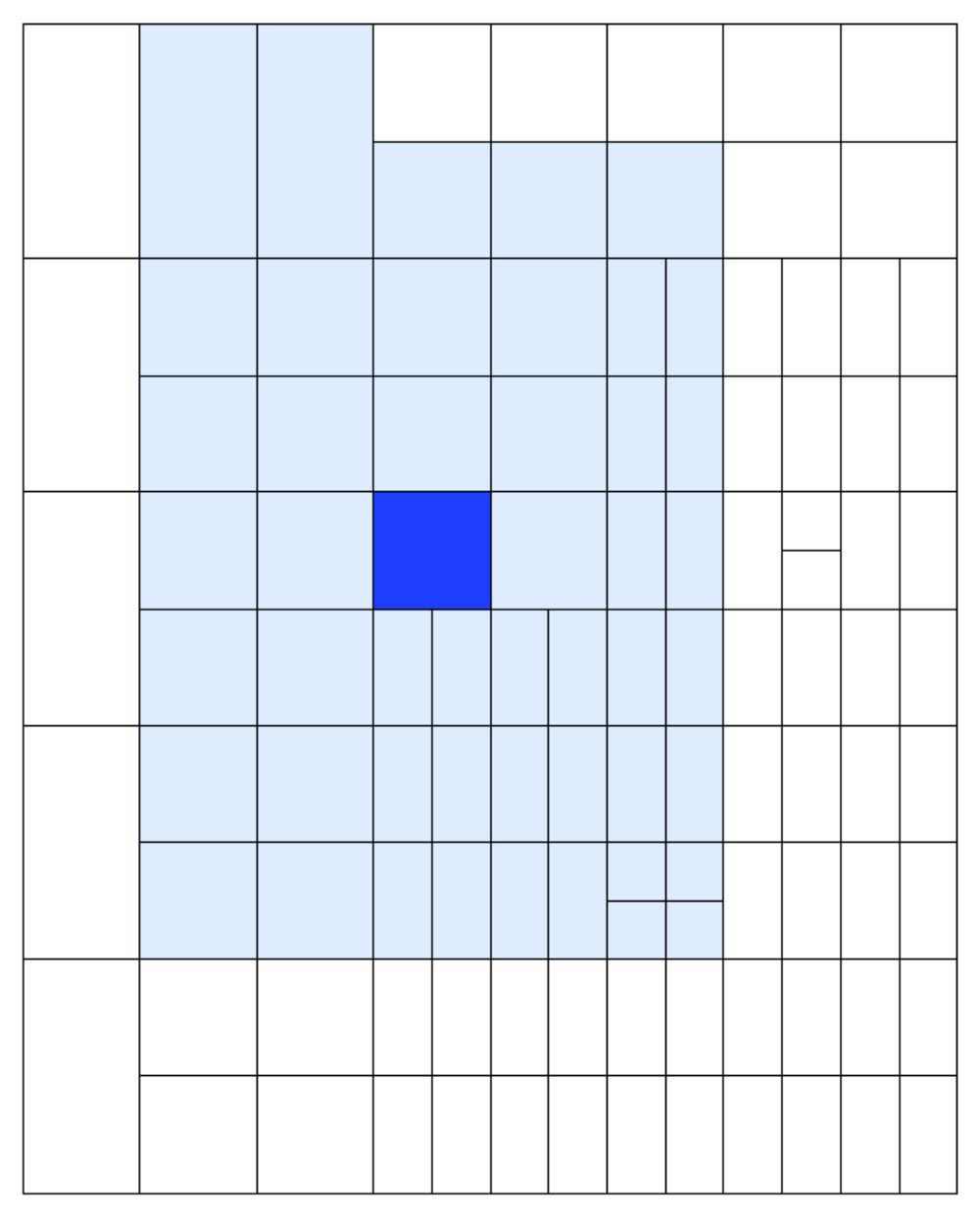}
\caption{Example of the patch $\pq\G K$ in a uniform mesh and in a non-uniform mesh for even $\ell(K)$ and $p=q=5$. $K$ is marked in blue, and $\pq\G K$ is highlighted in light blue.}
\label{fig: magic patch}
\end{figure}

In the subsequent definitions, we will give a detailed description of the elementary bisection steps and then present the new refinement algorithm. 

\begin{df}[Bisection of an element]
Given an arbitrary element $K=[\mu,\mu+\tilde\mu]\times[\nu,\nu+\tilde\nu]$, where $\mu, \nu, \tilde\mu,\tilde\nu\in\mathbb{R}$ and $\tilde\mu,\tilde\nu>0$, we define the operators
\begin{align*}
\bisect_\mathrm x(K) &\sei \bigl\{\,[\mu,\mu+\tfrac{\tilde\mu}2]\times[\nu,\nu+\tilde\nu],\enspace[\mu+\tfrac{\tilde\mu}2,\mu+\tilde\mu]\times[\nu,\nu+\tilde\nu]\,\bigr\}
\\\text{and}\enspace
\enspace\bisect_\mathrm y(K) &\sei \bigl\{\,[\mu,\mu+\tilde\mu]\times[\nu,\nu+\tfrac{\tilde\nu}2],\enspace[\mu,\mu+\tilde\mu]\times[\nu+\tfrac{\tilde\nu}2,\nu+\tilde\nu]\,\bigr\}.
\end{align*}
Note that $\bisect_\mathrm x$ adds an edge in $y$-direction, while $\bisect_\mathrm y$ adds an edge in $x$-direction. \end{df}
\begin{df}[Bisection]\label{df: bisection}%
Given a mesh $\G$ and an element $K\in\G$, we denote by $\bisect(\G,K)$ the mesh that results from a level-dependent bisection of $K$,
\begin{align*}
\bisect(\G,K)&\sei\G\setminus\{K\}\cup\child(K),\\
\text{with}\enspace\child(K)&\sei
\begin{cases}\bisect_\mathrm x(K)&\text{if $\ell(K)$ is even,}\\\bisect_\mathrm y(K)&\text{if $\ell(K)$ is odd.}\end{cases}
\end{align*}
\end{df}
\begin{df}[Multiple bisections]
We introduce the shorthand notation $\bisect(\G,\M)$ for the bisection of several elements $\M=\{K_1,\dots,K_J\}\subseteq\G$, defined by successive bisections in an arbitrary order,
\[\bisect(\G,\M)\sei\bisect(\bisect(\dots\bisect(\G,K_1),\dots),K_J).\]
\end{df}

We will now define the new refinement algorithm through the bisection of a superset $\clos_\G(\M)$ of the marked elements $\M$. In the remaining part of this section, we characterize the class of meshes generated by this refinement algorithm.

\begin{alg}[Closure]\label{alg: closure}
Given a mesh $\G$ and a set of marked elements $\M\subseteq\G$ to be bisected, the \emph{closure} $\clos_\G(\M)$ of $\M$ is computed as follows.
\begin{algorithmic}
\STATE $\Mtilde\sei \M$
\REPEAT
\FORALL{$K\in\Mtilde$}
\STATE $\Mtilde\sei\Mtilde\cup\bigl\{K'\in\pq\G K\mid \ell(K')<\ell(K)\bigr\}$
\ENDFOR
\UNTIL{$\Mtilde$ stops growing}
\RETURN $\clos_\G(\M)=\Mtilde$
\end{algorithmic}
\end{alg}

\begin{alg}[Refinement]\label{alg: refinement}

Given a mesh $\G$ and a set of marked elements $\M\subseteq\G$ to be bisected, $\refine(\G,\M)$ is defined by \[\refine(\G,\M)\sei\bisect(\G,\clos_\G(\M)).\]
\end{alg}
\begin{ex}
The Figures \ref{fig: refinement example 1}, \ref{fig: refinement example 2} and \ref{fig: refinement example 3}
illustrate three successive applications of Algorithm~\ref{alg: refinement} with $p=q=3$. In each case, only one element $K$ is marked. In the first case, the patch of $K$ is as fine as $K$ and hence no additional refinement is necessary. In the second case, one additional iteration of Algorithm~\ref{alg: closure} is needed to compute $\clos_\G(\{K\})$. In the third case, the Algorithm stops after three iterations.
\end{ex}
\begin{figure}
\centering
\includegraphics[width=.17\textwidth]{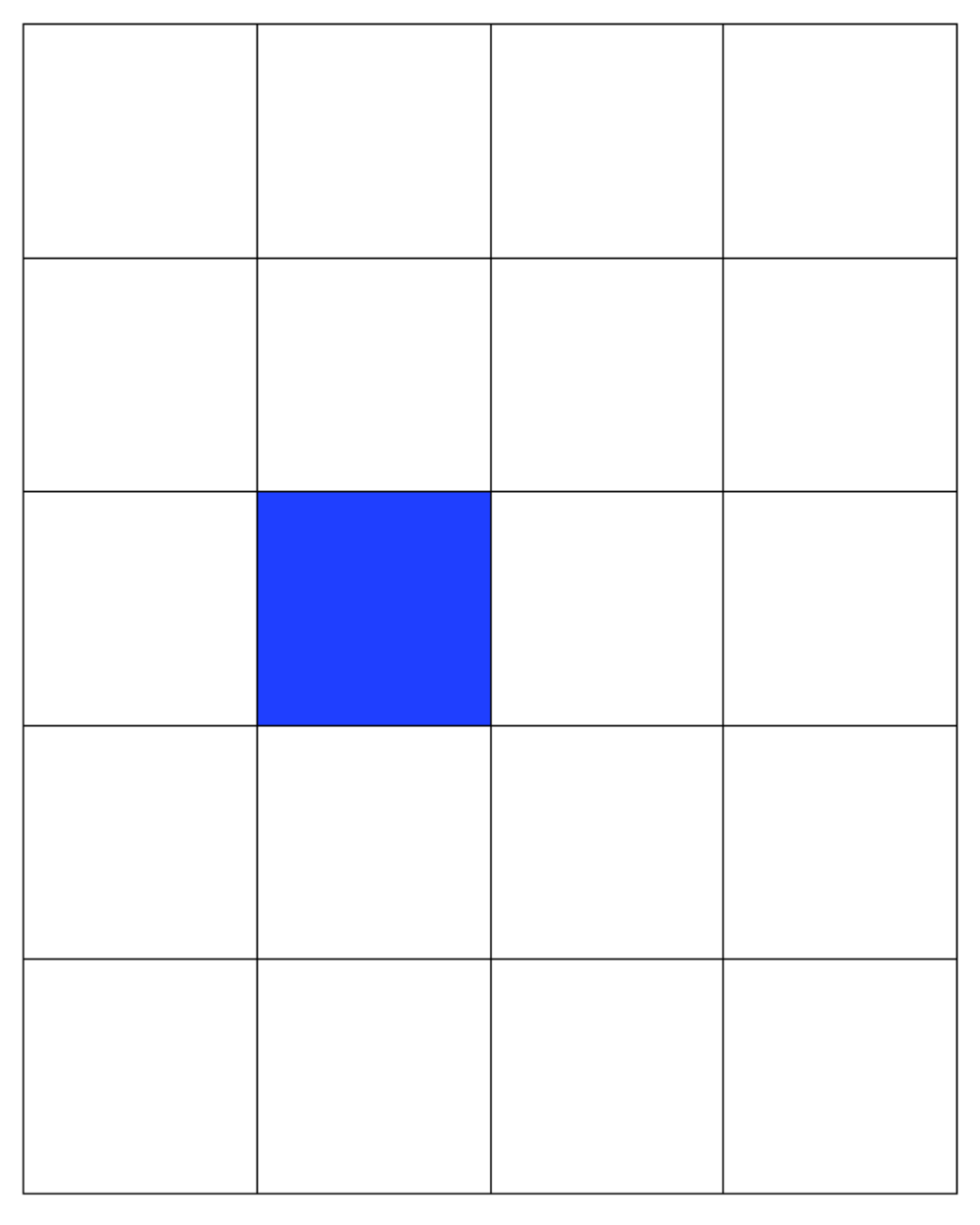}\raisebox{.10\textwidth}{\enspace$\rightarrow$\enspace}
\includegraphics[width=.17\textwidth]{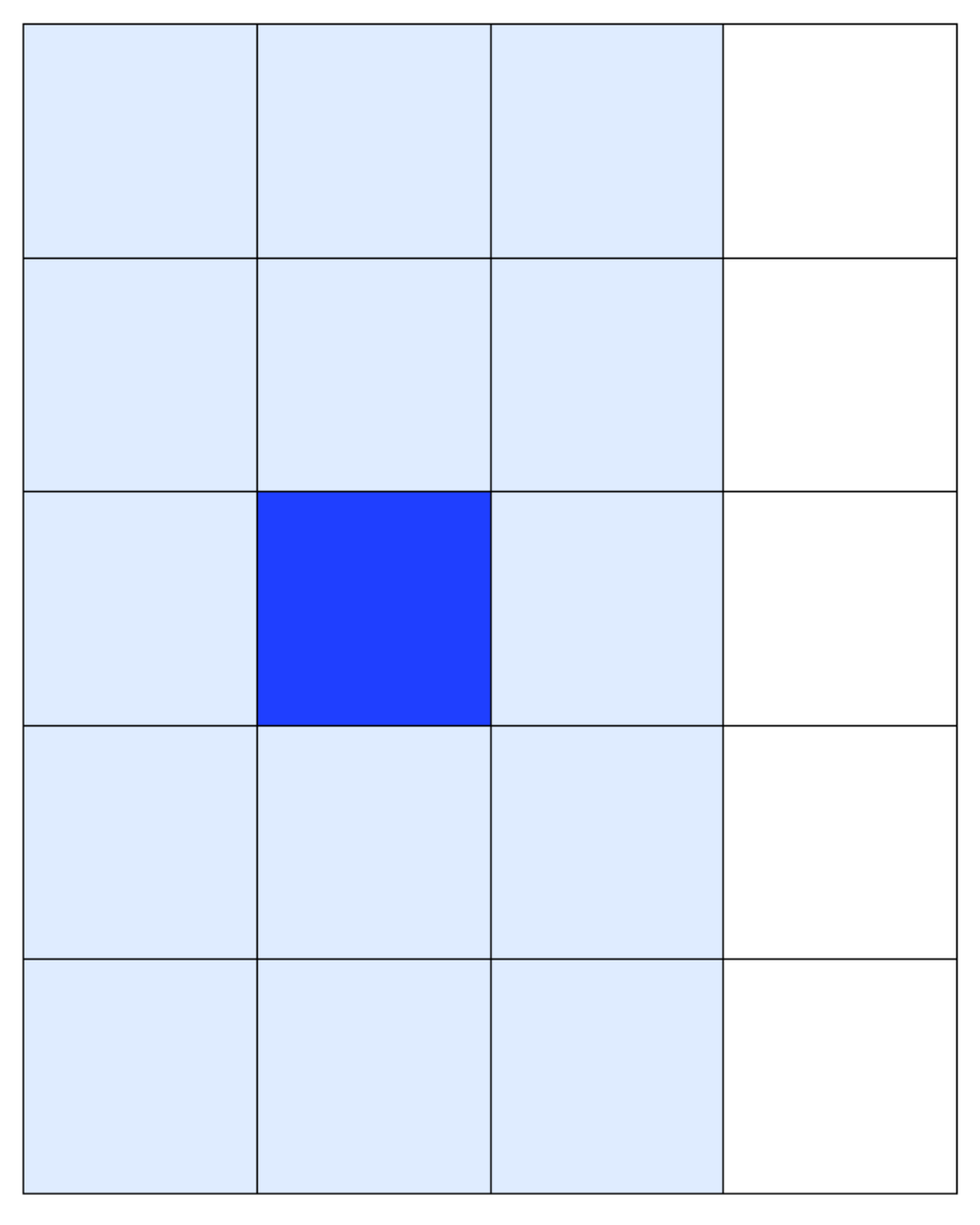}\raisebox{.10\textwidth}{\enspace$\rightarrow$\enspace}
\includegraphics[width=.17\textwidth]{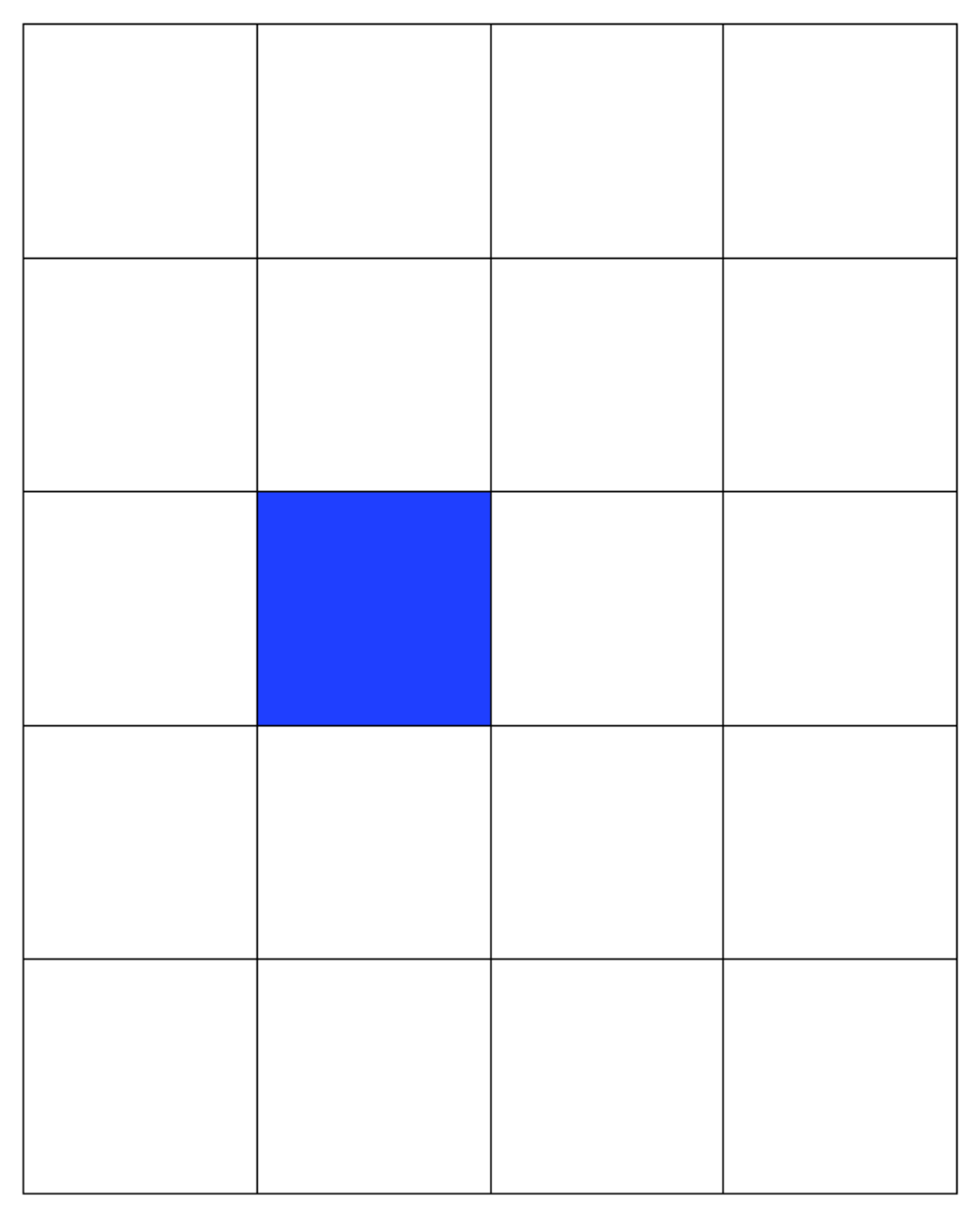}\raisebox{.10\textwidth}{\enspace$\rightarrow$\enspace}
\includegraphics[width=.17\textwidth]{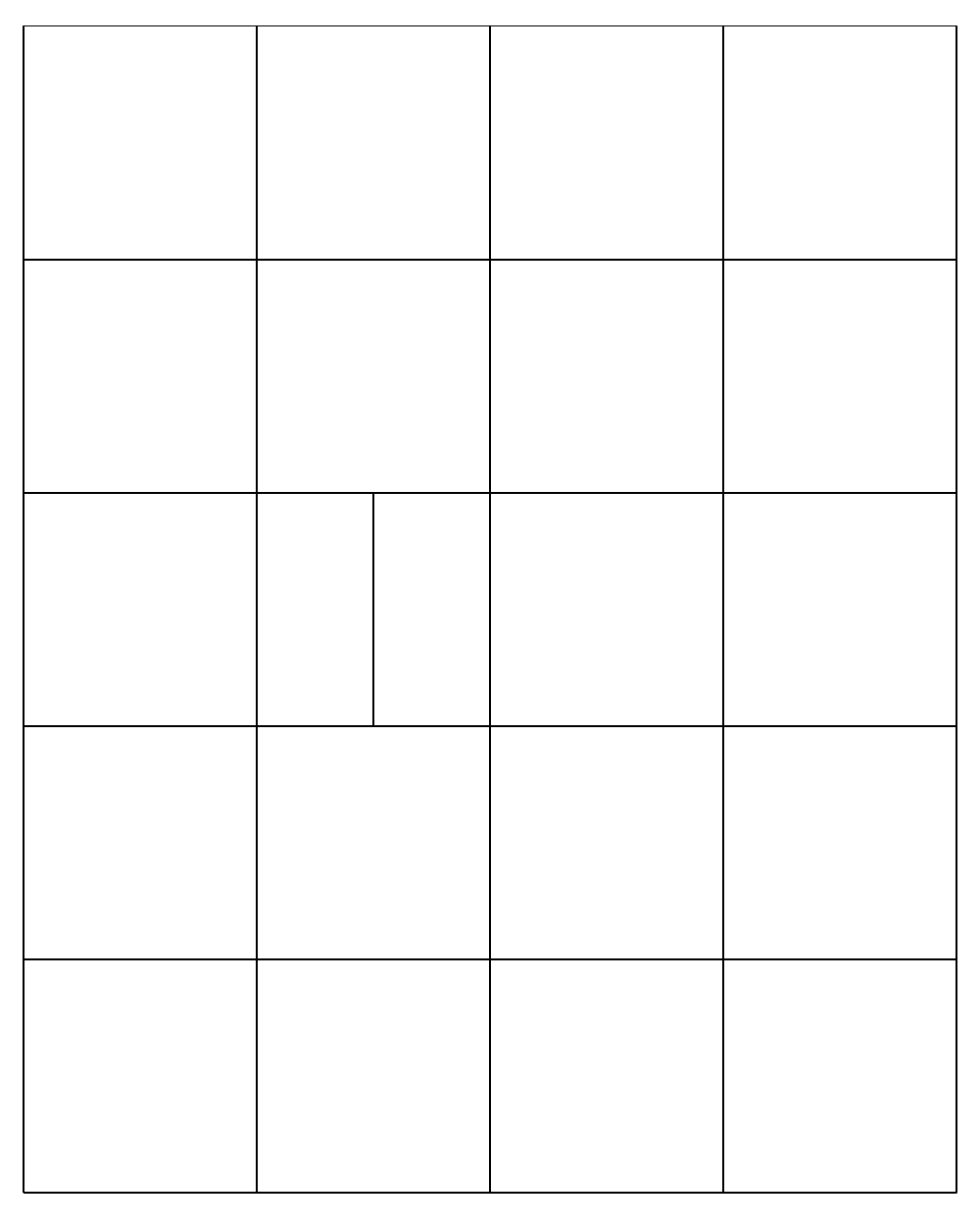}
\caption{First refinement example. The patch $\pq\G{K}$ (highlighted in light blue) is as fine as $K$. Consequently, Algorithm~\ref{alg: closure} stops after the first iteration.}
\label{fig: refinement example 1}
\end{figure}
\begin{figure}
\centering
\includegraphics[width=.17\textwidth]{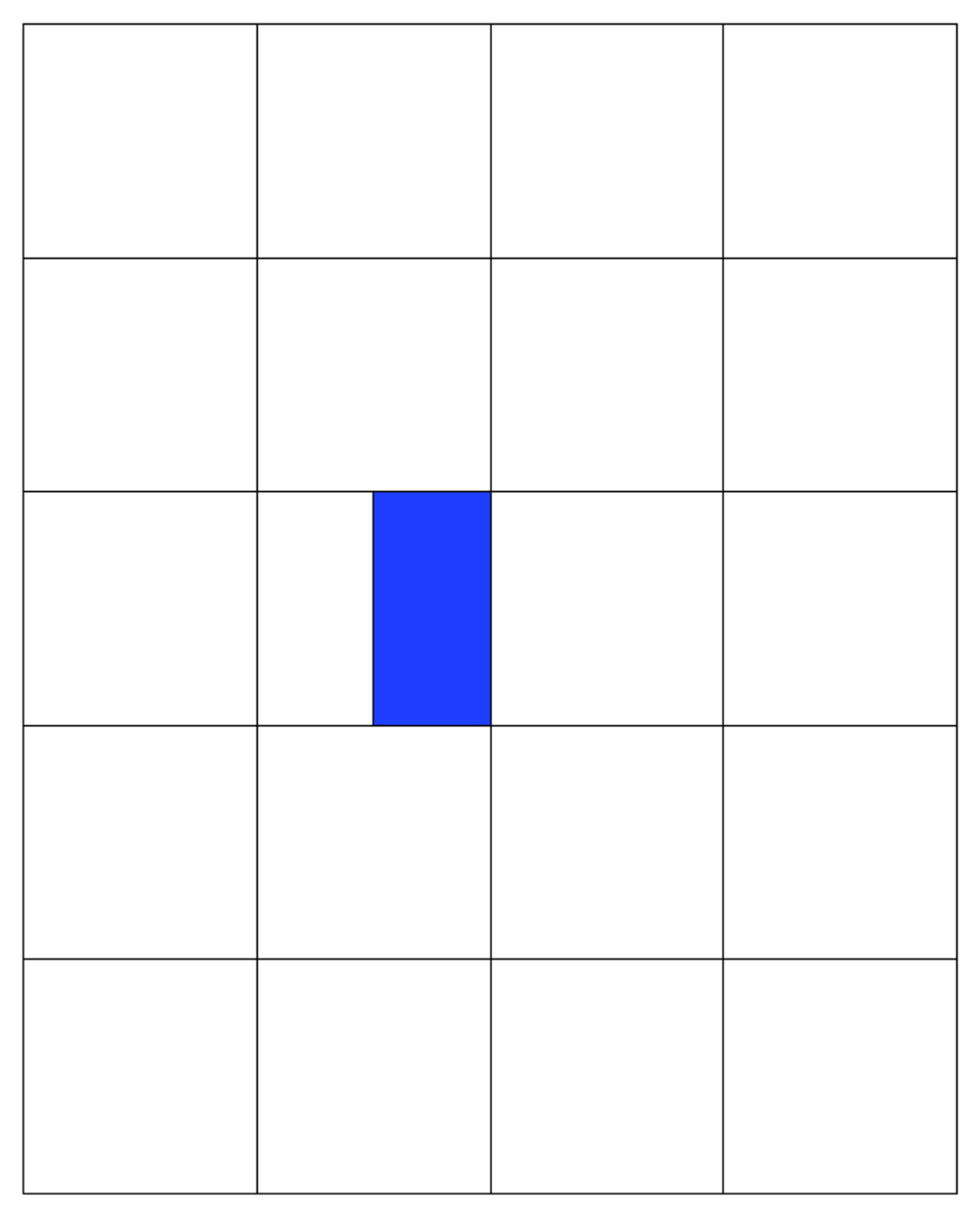}\raisebox{.10\textwidth}{\enspace$\rightarrow$\enspace}
\includegraphics[width=.17\textwidth]{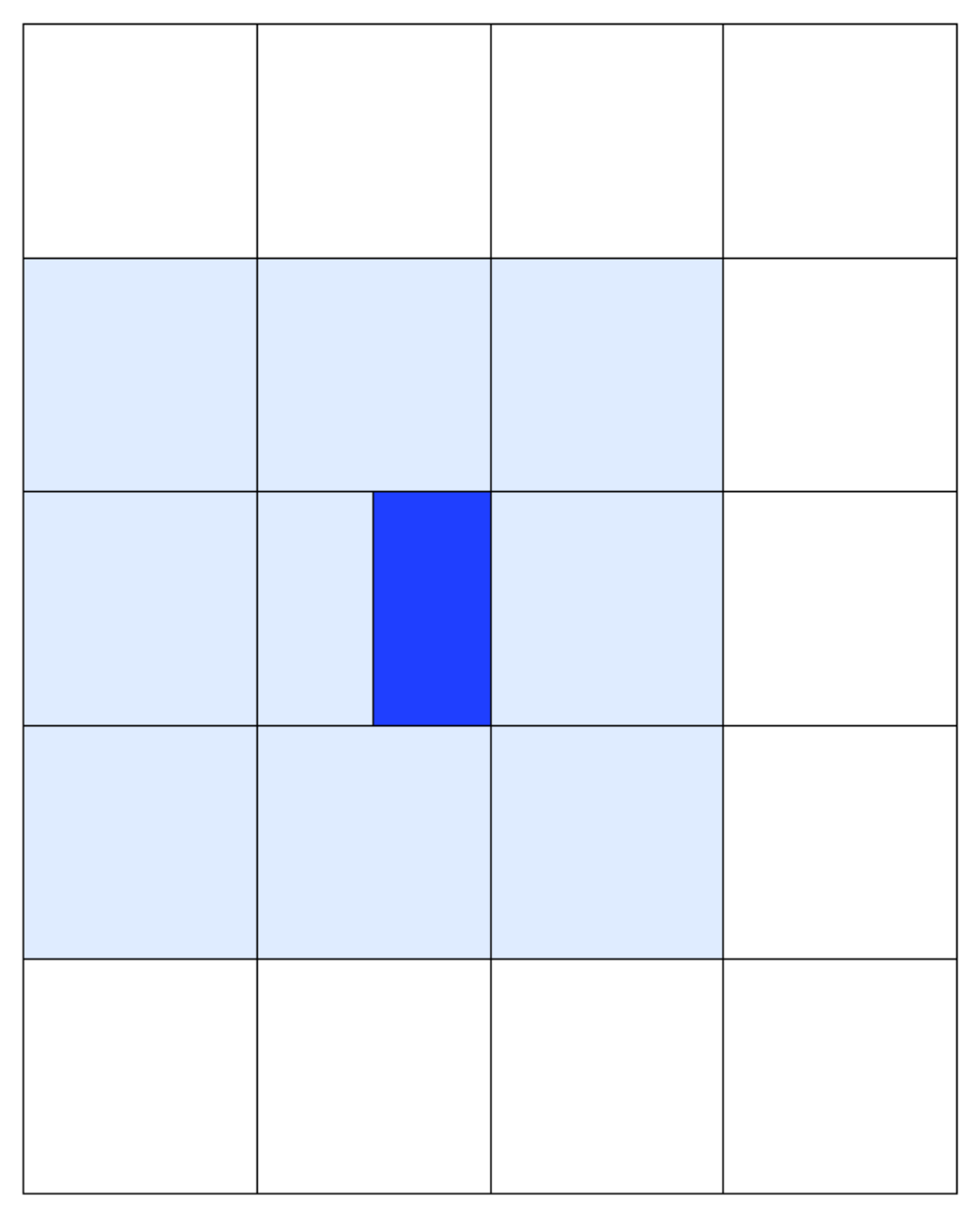}\raisebox{.10\textwidth}{\enspace$\rightarrow$\enspace}
\includegraphics[width=.17\textwidth]{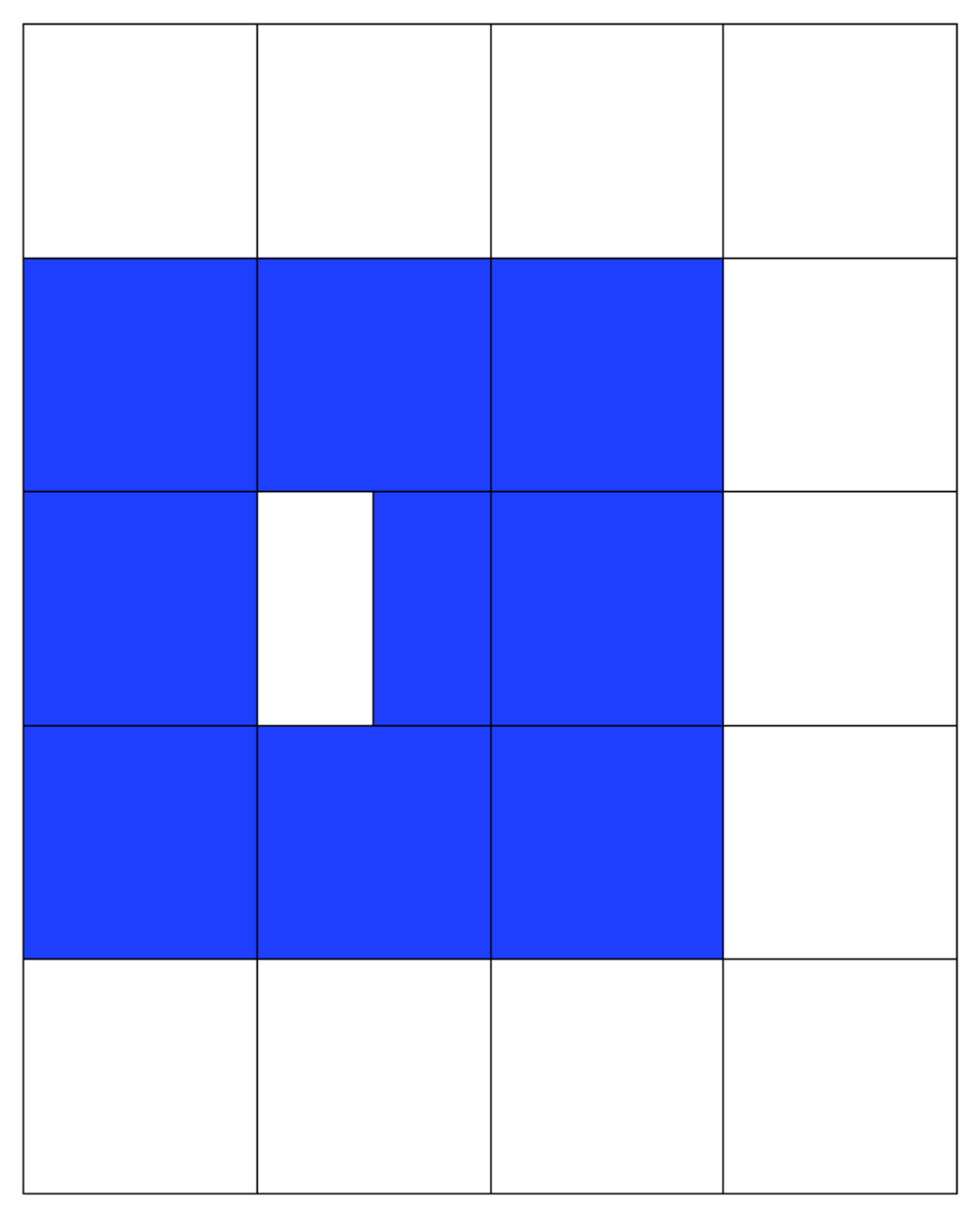}\\[1.5ex]\raisebox{.10\textwidth}{\enspace$\rightarrow$\enspace}
\includegraphics[width=.17\textwidth]{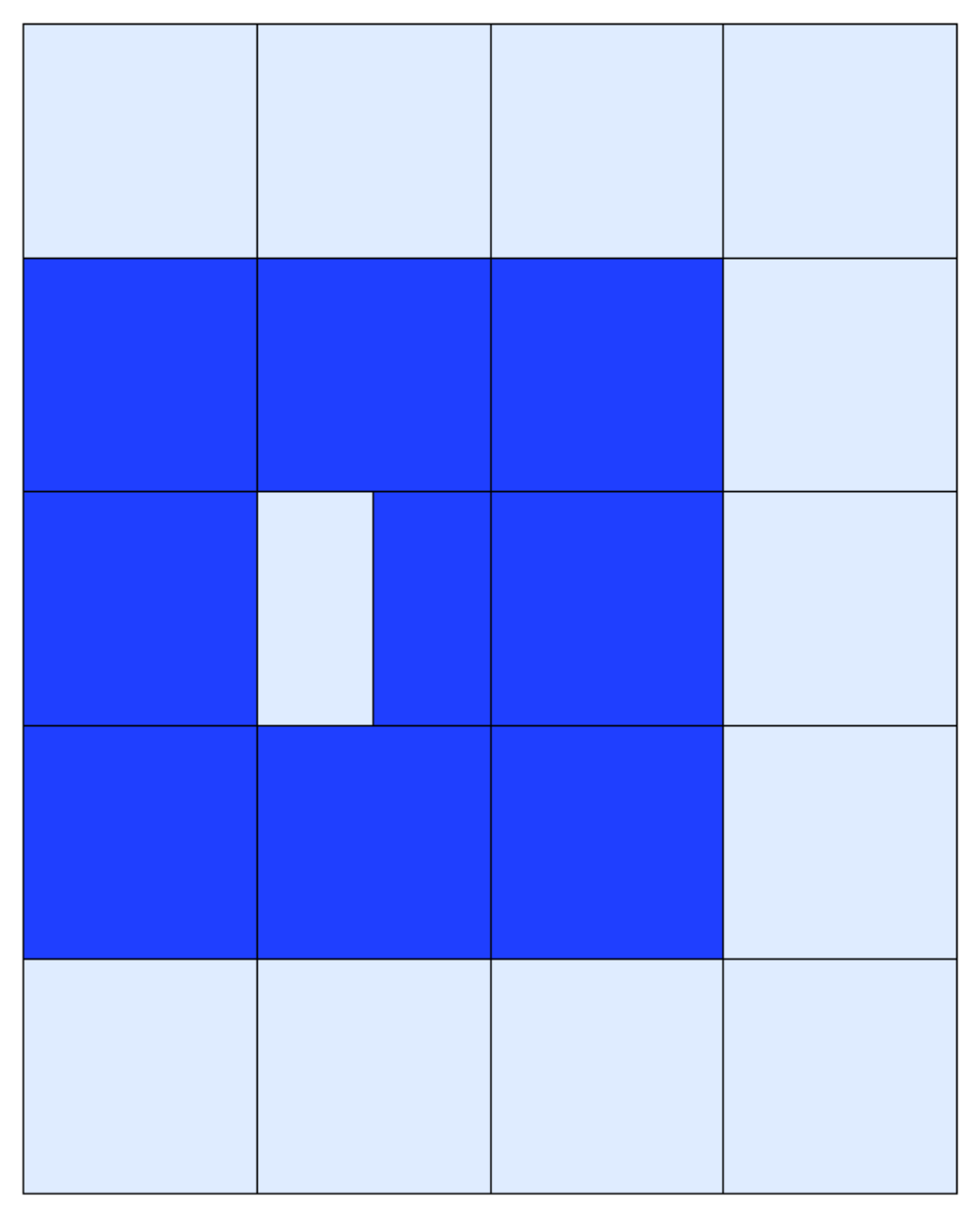}\raisebox{.10\textwidth}{\enspace$\rightarrow$\enspace}
\includegraphics[width=.17\textwidth]{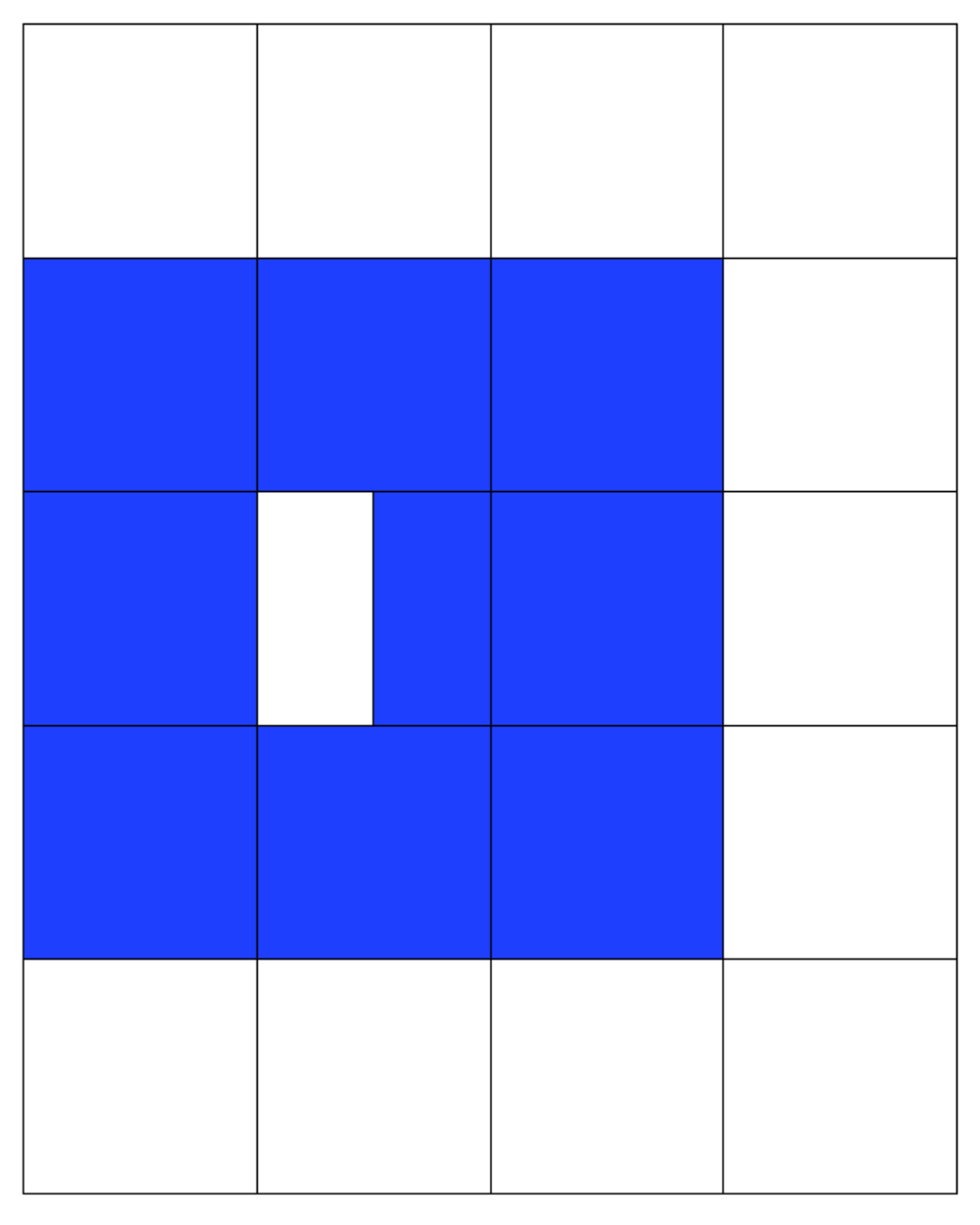}\raisebox{.10\textwidth}{\enspace$\rightarrow$\enspace}
\includegraphics[width=.17\textwidth]{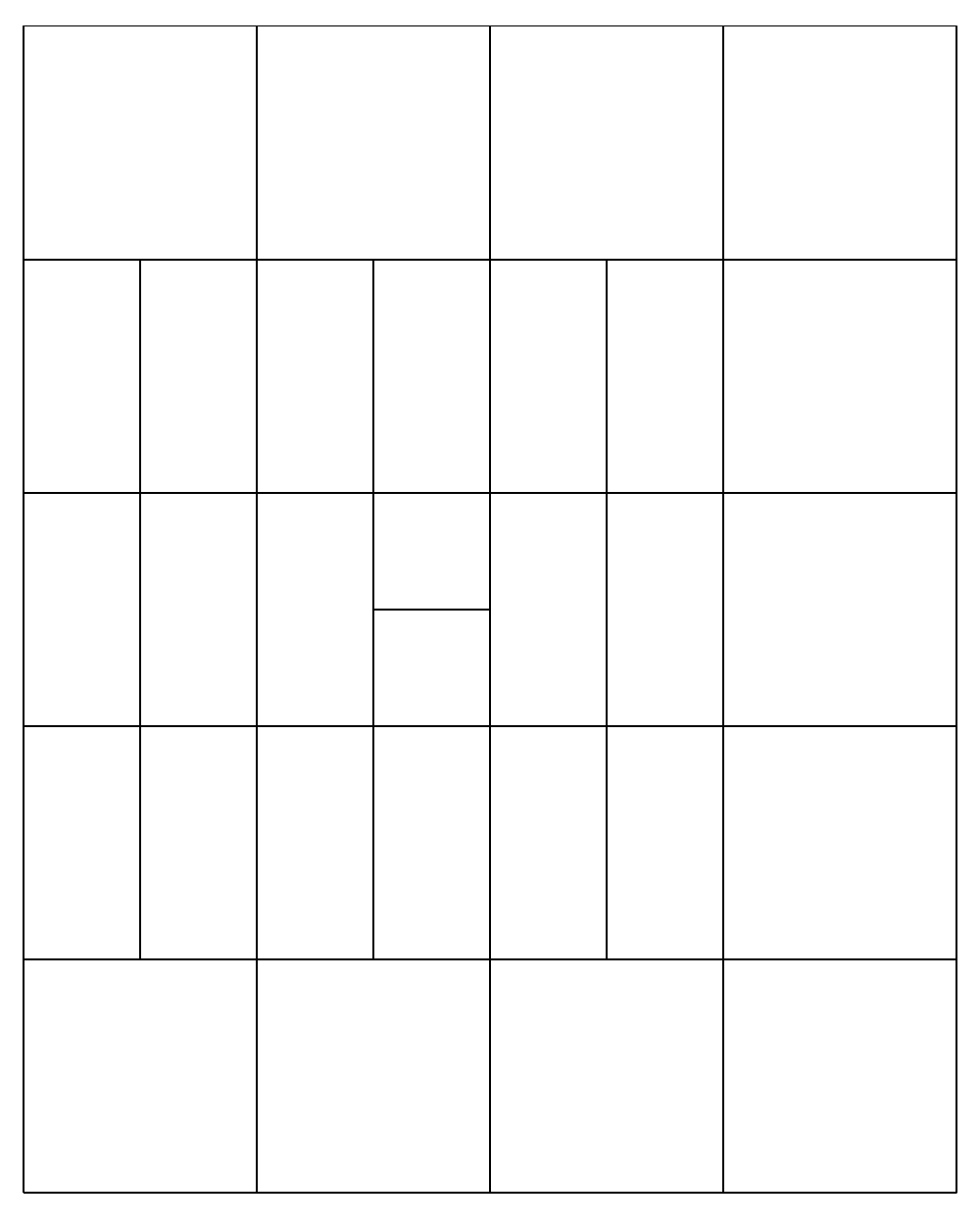}
\caption{Second refinement example. The patch $\pq\G{K}$ contains elements that are coarser than $K$. These are marked by Algorithm~\ref{alg: closure}. Then the algorithm checks their patches for even coarser elements, which do not exist. Hence Algorithm~\ref{alg: closure} stops after two iterations.}
\label{fig: refinement example 2}
\end{figure}
\begin{figure}
\centering
\includegraphics[width=.17\textwidth]{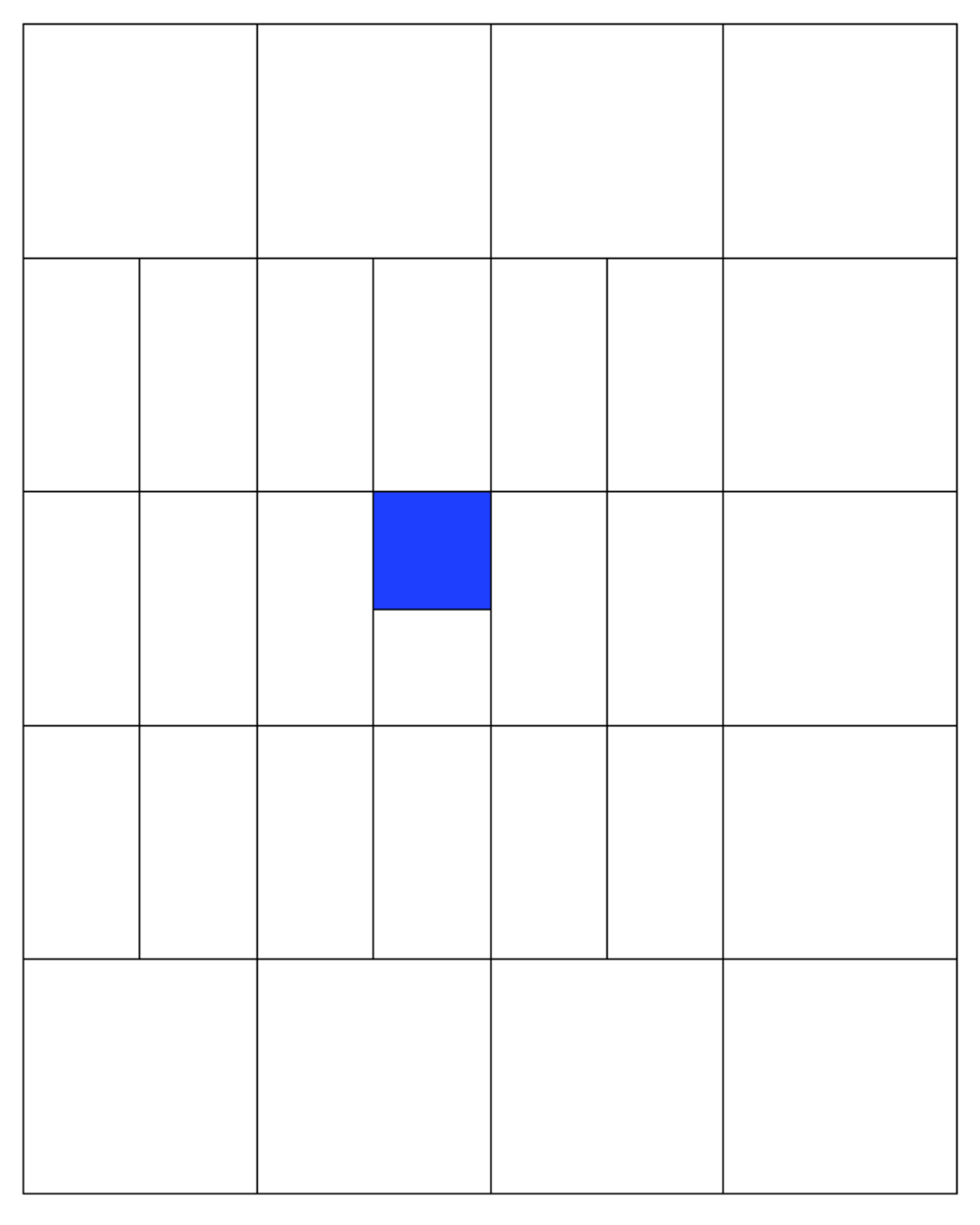}\raisebox{.10\textwidth}{\enspace$\rightarrow$\enspace}
\includegraphics[width=.17\textwidth]{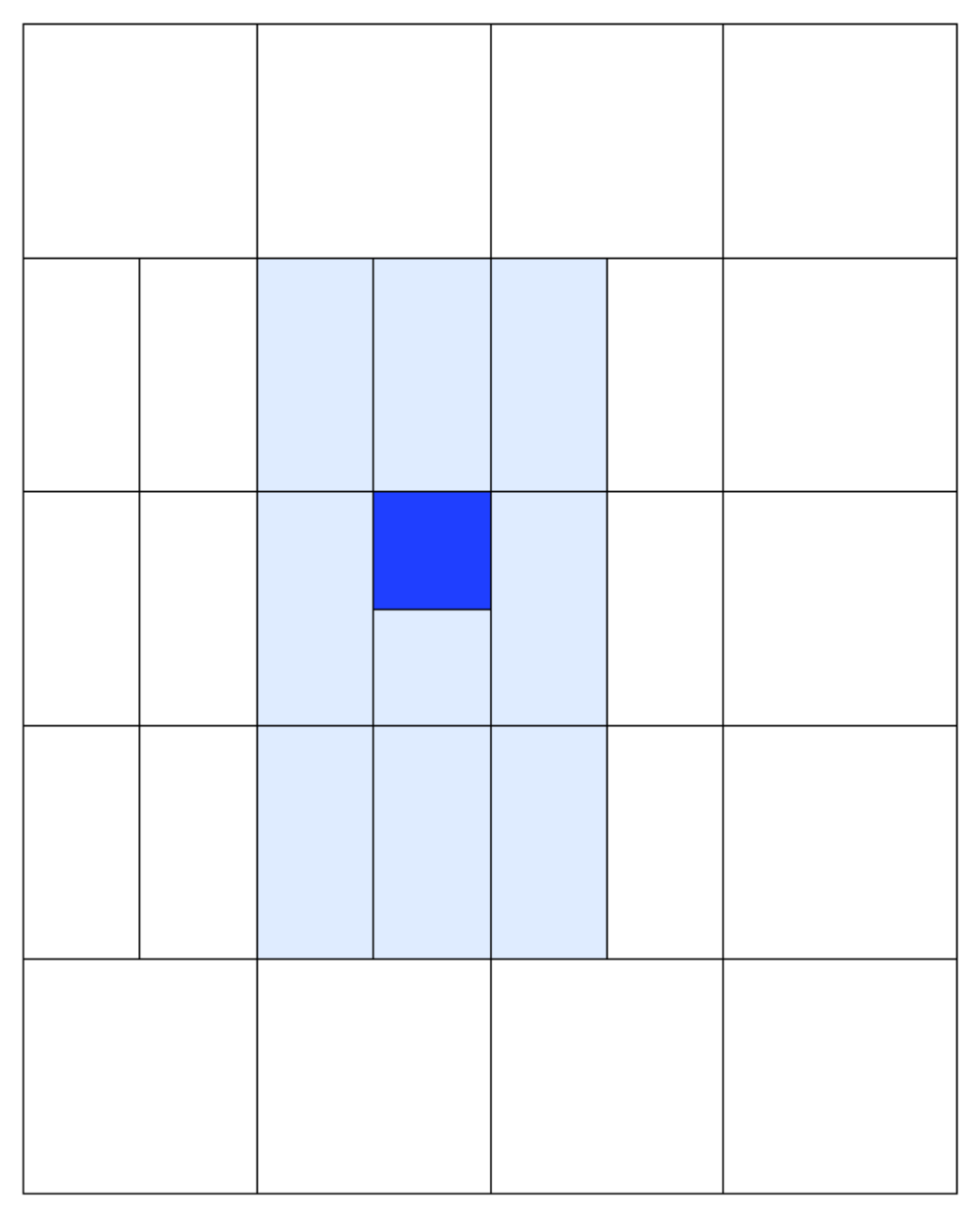}\raisebox{.10\textwidth}{\enspace$\rightarrow$\enspace}
\includegraphics[width=.17\textwidth]{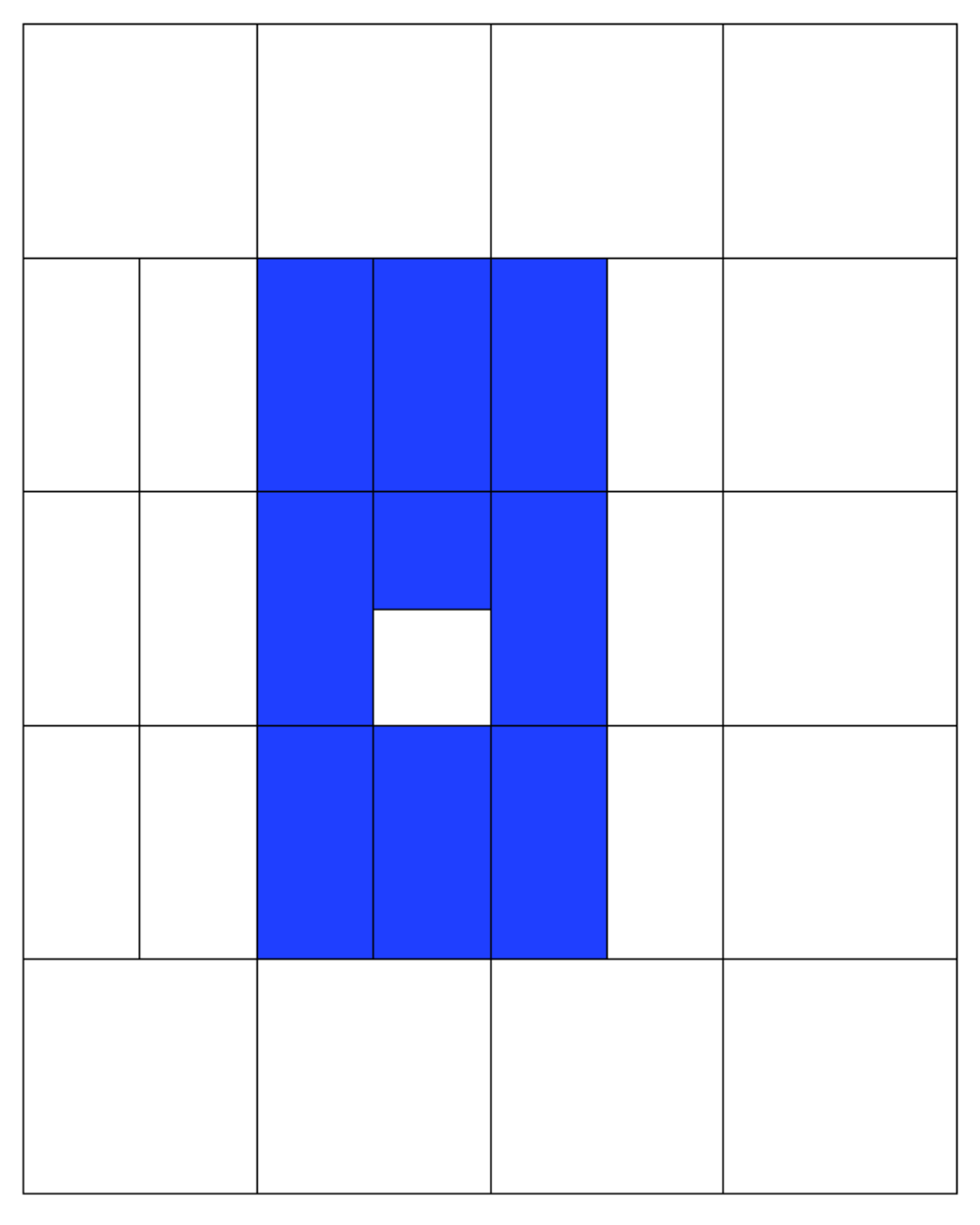}\raisebox{.10\textwidth}{\enspace$\rightarrow$\enspace}
\includegraphics[width=.17\textwidth]{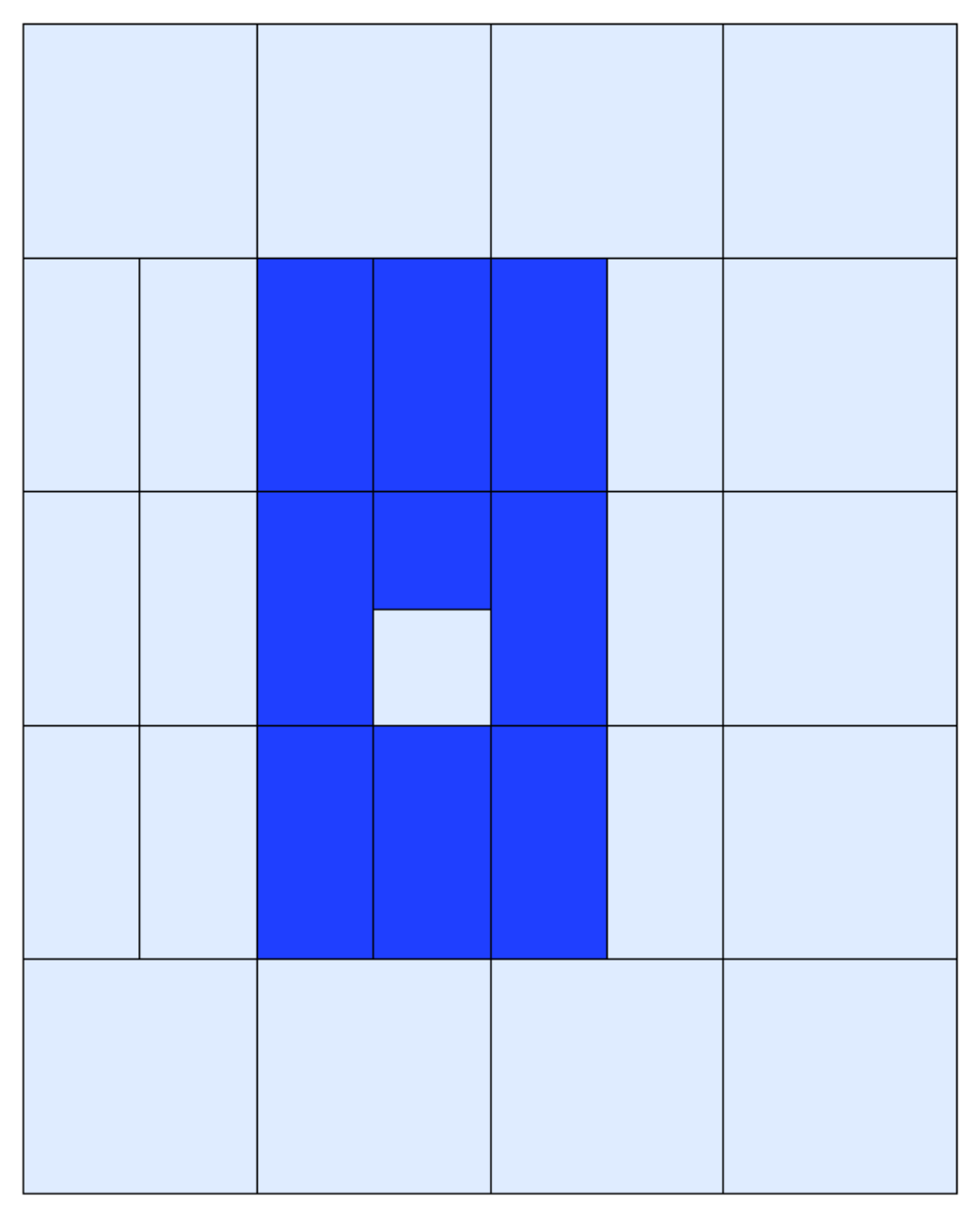}\\[1.5ex]\raisebox{.10\textwidth}{\enspace$\rightarrow$\enspace}
\includegraphics[width=.17\textwidth]{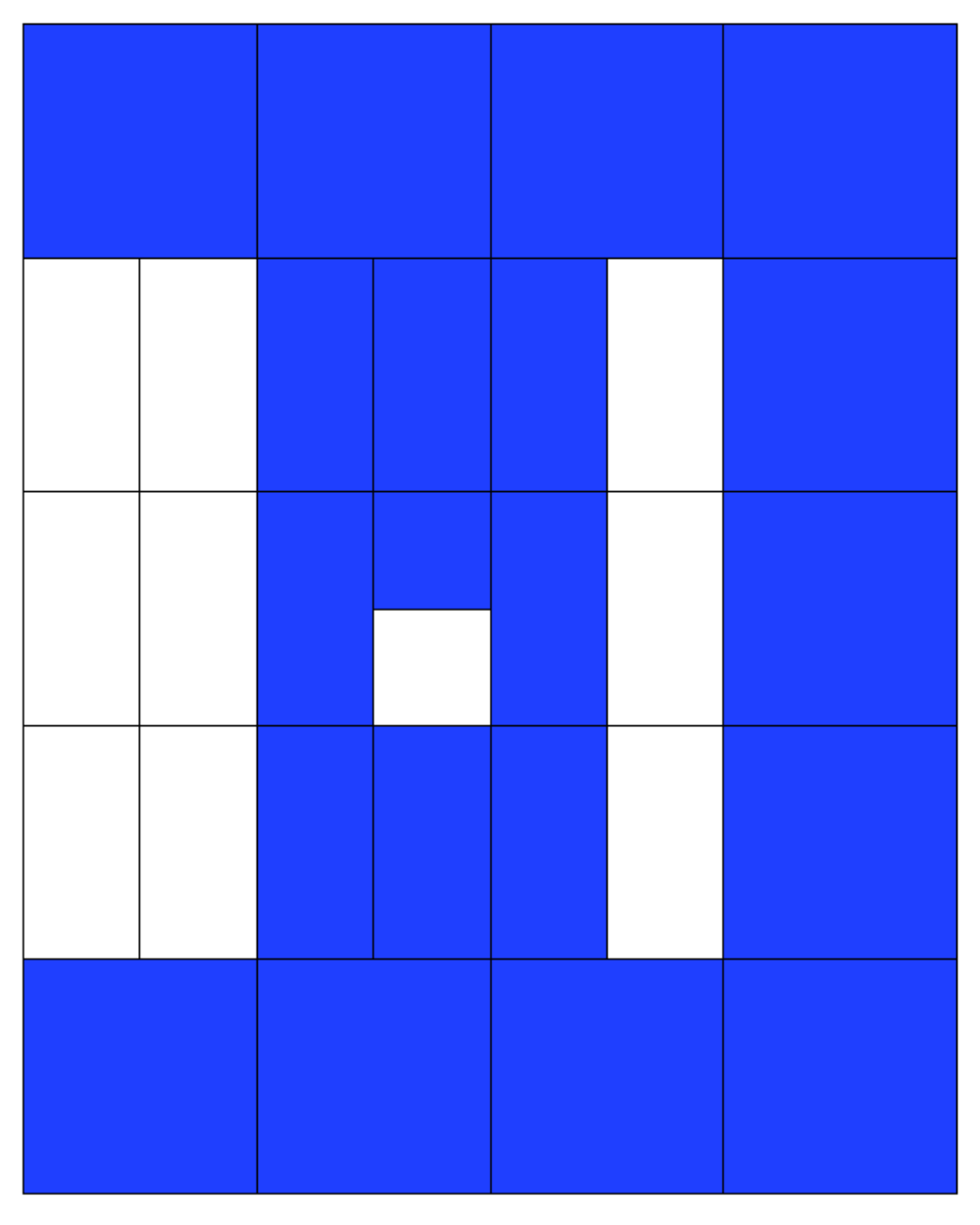}\raisebox{.10\textwidth}{\enspace$\rightarrow$\enspace}
\includegraphics[width=.17\textwidth]{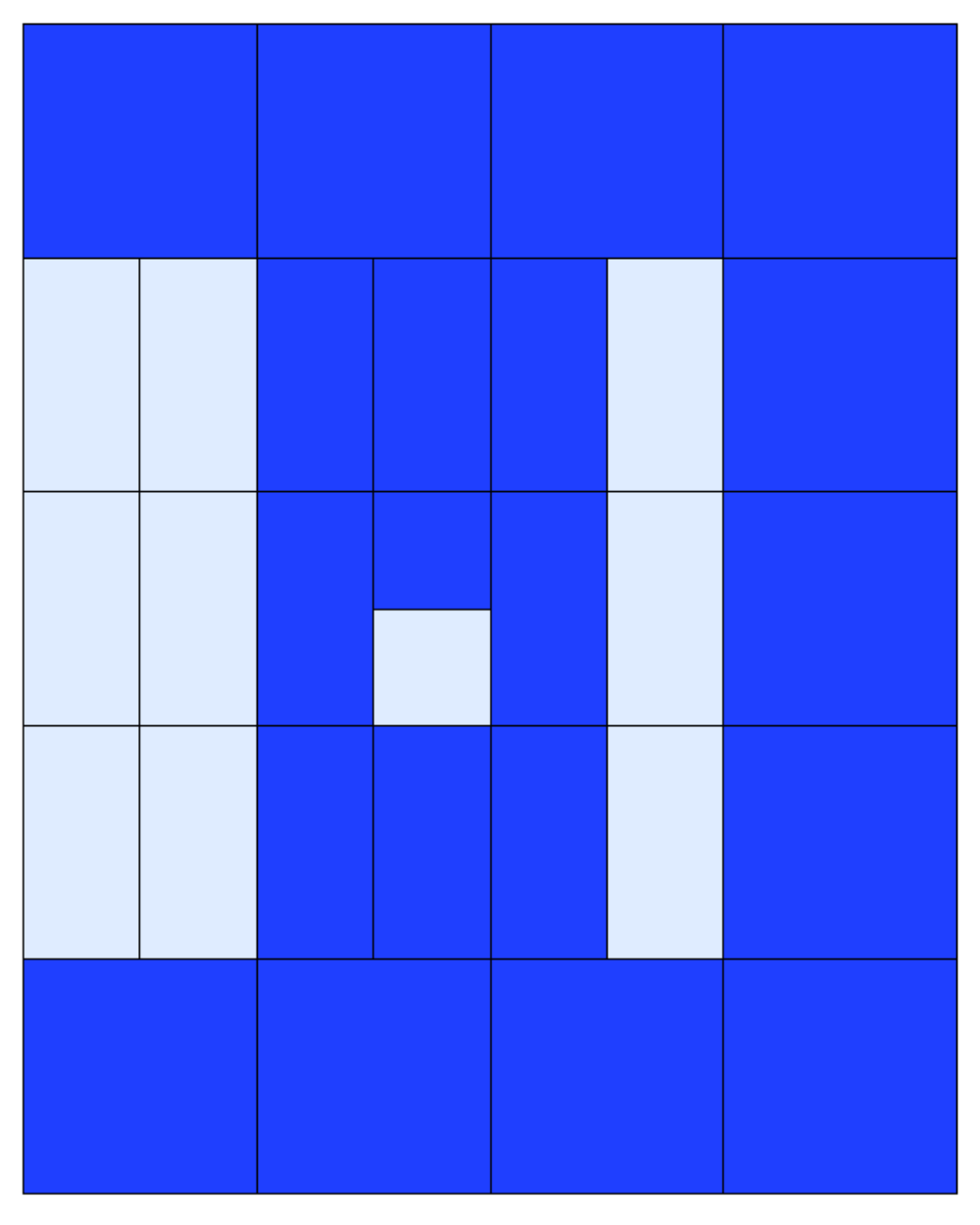}\raisebox{.10\textwidth}{\enspace$\rightarrow$\enspace}
\includegraphics[width=.17\textwidth]{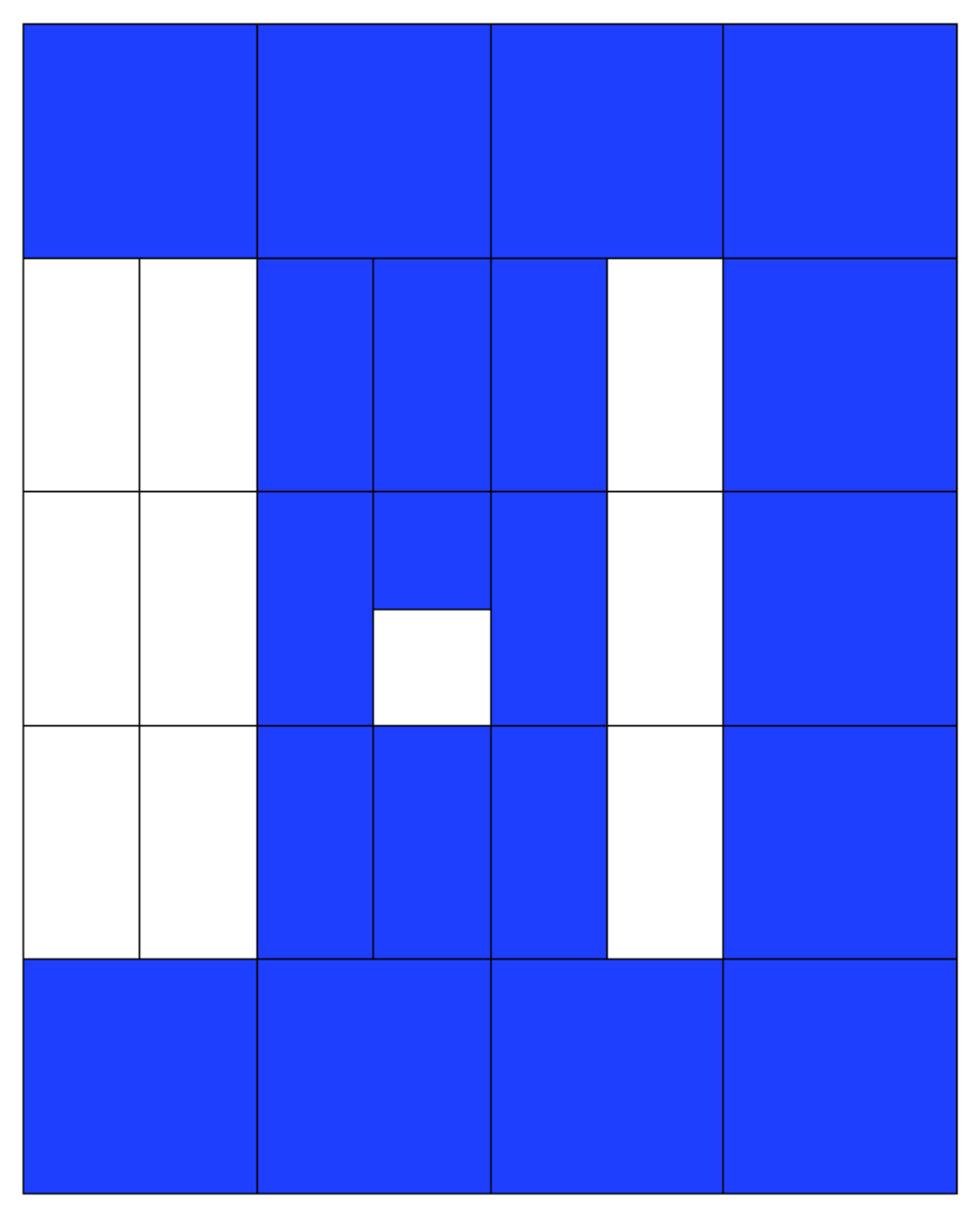}\raisebox{.10\textwidth}{\enspace$\rightarrow$\enspace}
\includegraphics[width=.17\textwidth]{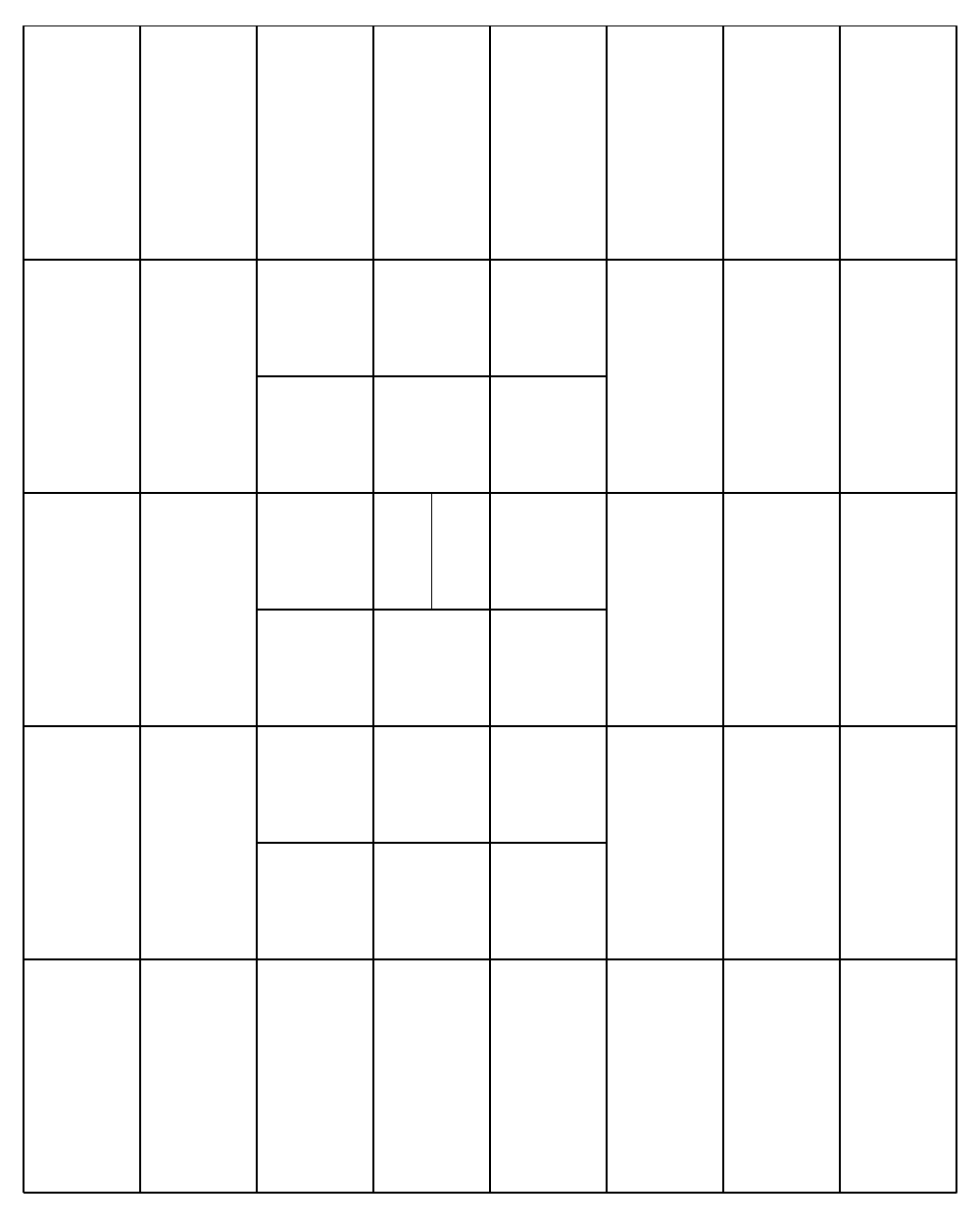}
\caption{Third refinement example. As in Figure~\ref{fig: refinement example 2}, Algorithm~\ref{alg: closure} marks coarser elements in the patch of the initially marked $K$. In this case, the computation of $\clos_\G(\{K\})$ involves three iterations of the algorithm.}
\label{fig: refinement example 3}
\end{figure}
In the subsequent definitions, we introduce a class of admissible meshes. We will then prove that Algorithm~\ref{alg: refinement} preserves admissibility.

\begin{df}[$(p,q)$-admissible bisections]\label{df: adm. bisection}%
Given a mesh $\G$ and an element $K\in\G$, the bisection of $K$ is called \emph{$(p,q)$-admissible} if all $K'\in\pq\G K$ satisfy $\ell(K')\ge\ell(K)$.

In the case of several elements $\M=\{K_1,\dots,K_J\}\subseteq\G$, the bisection $\bisect(\G,\M)$ is $(p,q)$-admissible if there is an order $(\sigma(1),\dots,\sigma(J))$ (this is, if there is a permutation $\sigma$ of $\{1,\dots,J\}$) such that \[\bisect(\G,\M)=\bisect(\bisect(\dots\bisect(\G,K_{\sigma(1)}),\dots),K_{\sigma(J)})\] is a concatenation of $(p,q)$-admissible bisections.
\end{df}
\begin{df}[Admissible mesh]\label{df: admissible mesh}
A refinement $\G$ of $\G_0$ is \emph{$(p,q)$-admissible} if there is a sequence of meshes $\G_1,\dots,\G_J=\G$ and markings $\M_j\subseteq\G_j$ for $j=0,\dots,J-1$, such that $\G_{j+1}=\bisect(\G_j,\M_j)$ is an $(p,q)$-admissible bisection for all $j=0,\dots,J-1$. The set of all $(p,q)$-admissible meshes, which is the initial mesh and its $(p,q)$-admissible refinements, is denoted by $\A$. For the sake of legibility, we write `admissible' instead of `$(p,q)$-admissible' throughout the rest of this paper.
\end{df}
\begin{rem}
This definition refers to the understanding of `admissible meshes' in FE analysis. It does \emph{not} match the definitions of admissible meshes from \cite{AS-Tsplines-of-arb-degree,ASTS-characterization}.
\end{rem}

\begin{prp}\label{prp: ref works}
Any admissible mesh $\G$ and any set of marked elements $\M\subseteq\G$ satisfy $\refine(\G,\M)\in\A$.
\end{prp}
The proof of Proposition~\ref{prp: ref works} given at the end of this section relies on the subsequent results.

\begin{lma}[local quasi-uniformity]\label{lma: levels change slowly}
Given  $K\in\G\in\A$, 
any $K'\in\pq\G K$ satisfies $\ell(K')\ge\ell(K)-1$.
\end{lma}
\begin{proof}
For $\ell(K)=0$, the assertion is always true. For $\ell(K)>0$, consider the parent $\hat K$ of $K$ (i.e., the unique element $\hat K\in\tcup\A$ with $K\in\child(\hat K)$). Since $K$ results from the bisection of $\hat K$, we also have that
\begin{align*}
d(K)\sei\Dist(K,\hat K)&=\begin{cases}(2^{-(\ell(\hat K)+4)/2},0)&\text{if $\ell(\hat K)$ is even,}\\(0,2^{-(\ell(\hat K)+3)/2})&\text{if $\ell(\hat K)$ is odd.}\end{cases}\\
&=\begin{cases}(0,2^{-(\ell(K)+2)/2})&\text{if $\ell(K)$ even,}\\(2^{-(\ell(K)+3)/2},0)&\text{if $\ell(K)$ odd.}\end{cases}
\end{align*}
Since $\G$ is admissible, there are admissible meshes $\G_0,\dots,\G_J=\G$ and some $j\in\{0,\dots,J-1\}$ such that $K\in\G_{j+1}=\bisect(\G_j,\{\hat K\})$.
The admissibility $\G_{j+1}\in\A$ implies that any $K'\in\pq{\G_j}{\hat K}$ satisfies $\ell(K')\ge\ell(\hat K)=\ell(K)-1$.
Since levels do not decrease during refinement, we get
\begin{alignat}{2}
\notag \ell(K)-1&\le\min\bigl\{\ell(K')\mid K'\in\G_j&&\text{ and }\Dist(\hat K,K')\le\D(\ell(\hat K))\bigr\}\\
\notag &\le\min\bigl\{\ell(K')\mid K'\in\G&&\text{ and }\Dist(\hat K,K')\le\D(\ell(\hat K))\bigr\}\\
\notag &=\min\bigl\{\ell(K')\mid K'\in\G&&\text{ and }\Dist(\hat K,K')\le\D(\ell(K)-1)\bigr\}\\
\label{eq: levels change slooowly} &\le\min\bigl\{\ell(K')\mid K'\in\G&&\text{ and }\Dist(K,K')+d(K)\le\D(\ell(K)-1)\bigr\}.\end{alignat}
One easily computes $\D(\ell(K)-1)-d(K)>\D(\ell(K))$, which concludes the proof.
\end{proof}

\begin{crl}\label{crl: old magic patch}
Let $K\in\G\in\A$ and
\begin{align*}
\U K&\sei\{x\in\barOmega\mid\Dist(K,x)\le\D(\ell(K))\},
\shortintertext{then}
\pq\G K &= \{K'\in\G\mid\lvert K'\cap\U K\rvert>0\}.
\end{align*}
\end{crl}
\begin{proof}
This is a consequence of Lemma~\ref{lma: levels change slowly} in the strong version \eqref{eq: levels change slooowly} that involves a bigger patch of $K$.
\end{proof}

\begin{proof}[Proof of Proposition~\ref{prp: ref works}]
Given the mesh $\G\in\A$ and marked elements $\M\subseteq\G$ to be bisected, we have to show that there is a sequence of meshes that are subsequent admissible bisections, with $\G$ being the first and $\refine(\G,\M)$ the last mesh in that sequence.
Set $\Mtilde\sei\clos_\G(\M)$ and
\begin{alignat}{2}
\notag\overline L&\sei\max\ell(\Mtilde),\quad\underline L\sei\min\ell(\Mtilde)&&\\
\notag\M_j&\sei\bigl\{K\in\Mtilde\mid\ell(K)=j\bigr\}&&\text{for}\enspace j=\underline L,\dots,\overline L\\
\label{eq: ref works 1}\G_{\underline L}&\sei\G,\quad \G_{j+1}\sei\bisect(\G_j,\M_j)&\enspace&\text{for}\enspace j=\underline L,\dots,\overline L.
\end{alignat}
It follows that $\refine(\G,\M)=\G_{\overline L+1}$. We will show by induction over $j$ that all bisections in  \eqref{eq: ref works 1} are admissible. 

For the first step $j=\underline L$, we know $\{K'\in\Mtilde\mid\ell(K')<\underline L\}=\emptyset$, and by construction of $\Mtilde$ that for each $K\in\Mtilde_{\underline L}$ holds  
$\{K'\in\pq\G K\mid\ell(K')<\ell(K)\}\subseteq\Mtilde$.
Together with $\ell(K)=\underline L$ follows for any $K\in\Mtilde_{\underline L}$ that  there is no $K'\in\pq\G K$ with $\ell(K')<\ell(K)$. This is, the bisections of all $K\in\Mtilde_{\underline L}$ are admissible independently of their order and hence $\bisect(\G_{\underline L},\Mtilde_{\underline L})$ is admissible.

Consider an arbitrary step $j\in\{\underline L,\dots,\overline L\}$ and assume that $\G_{\underline L},\dots,\G_j$ are admissible meshes. 
Assume for contradiction that there is $K\in\M_j$ of which the bisection is not admissible, i.e., there exists $K'\in\smash{\pq{\G_j}K}$ with $\ell(K')<\ell(K)$ and consequently $K'\notin\Mtilde$, because $K'$ has not been bisected yet. It follows from the closure Algorithm~\ref{alg: closure} that $K'\notin\G$. Hence,  there is $\hat K\in\G$ such that $K'\subset\hat K$. 
We have $\ell(\hat K)<\ell(K')<\ell(K)$, which implies $\ell(\hat K)<\ell(K)-1$. Note that $K\in\G$ because $\M_j\subseteq\Mtilde\subseteq\G$.
Moreover, from $K'\subset\hat K$ and $K'\in\pq{\G_j}K$ it follows with Corollary~\ref{crl: old magic patch} that $\hat K\in\pq\G K$.
Together with $\ell(\hat K)<\ell(K)-1$, Lemma~\ref{lma: levels change slowly} implies that $\G$ is not admissible, which contradicts the assumption.
\end{proof}

\section{Analysis-Suitability}\label{sec: AS}
In this section, we give a brief review on the concept of Analysis-Suitability, using the notation from \cite{AS-Tsplines-of-arb-degree}. We prove that all admissible meshes (in the sense of Definition~\ref{df: admissible mesh}) are analysis-suitable and hence provide linearly independent T-spline blending functions. In this paper, we omit the definition of the T-spline blending functions and details on their linear independence. We refer the reader to \cite{lin-ind-Tsplines,AS-Tsplines-are-DC} and, in particular for the case of non-cubic T-splines, \cite{AS-Tsplines-of-arb-degree}.

\begin{df}[Active nodes]
Consider an admissible mesh $\G\in\A$. The set of vertices (\emph{nodes}) of $\G$ is denoted by $\N$.
We define the \emph{active region} \[\AR\sei\bigl[\ceilfrac p2,M-\ceilfrac p2\bigr]\times\bigl[\ceilfrac q2,N-\ceilfrac q2\bigr]\]
and the set of \emph{active nodes} $\N_A\sei\N\cap\AR$. 
\end{df}
To each active node $T$, we associate local index vectors $\xx(T)$ and $\yy(T)$ that are defined below, depending on the mesh in the neighbourhood of $T$. These local index vectors are used to construct a tensor-product B-spline $B_T$, referred to as \emph{T-spline blending function}.

\begin{df}[Skeleton]
We denote  by $\hsk$ (resp.\ $\vsk$) the horizontal (resp.\ vertical) skeleton, which is the union of all horizontal (resp.\ vertical) edges.
Note that $\hsk\cap\vsk=\N$.
\end{df}

\begin{df}[Global index sets]\label{df: global indices}%
For any $y$ in the closed interval $\bigl[\ceilfrac q2,N-\ceilfrac q2\bigr]$, we set
\begin{alignat*}{2}
\XX(y) &\sei \bigl\{z\in[0,&M]&\mid(z,y)\in\vsk\bigr\},\\[-.5em]
\intertext{and for any $x\in\bigl[\ceilfrac p2,M-\ceilfrac p2\bigr]$,}\\[-2em]
\YY(x) &\sei \bigl\{z\in[0,&N]&\mid(x,z)\in\hsk\bigr\}.
\end{alignat*}
Note that in an admissible mesh, the entries $\bigl\{ 0,\dots,\ceilfrac p2-1,\enspace M-\ceilfrac p2+1,\dots,M\bigr\}$ are always included in $\XX(y)$ (and analogously for $\YY(x)$).
\end{df}
\begin{df}[{T-junction extension \cite[Section~2.1]{AS-Tsplines-of-arb-degree}}]\label{df: TJ-extensions}%
We denote by $\mathcal T\subset\N_A$ the set of all active nodes with valence three (i.e., active nodes that are endpoints of exactly three edges) and refer to them as \emph{T-junctions}.
Following the literature \cite{lin-ind-Tsplines,AS-Tsplines-are-DC}, we adopt the notation ${\bot},{\top},{\vdash},{\dashv}$ to indicate the four possible orientations of the T-junctions. T-junctions of type ${\dashv}$ and ${\vdash}$ (${\bot},{\top}$, respectively) and their extensions are called horizontal (vertical, resp.). For the sake of simplicity, let us consider a T-junction $T=(t_1,t_2)\in\mathcal T$ of type ${\dashv}$. Clearly, $t_1$ is one of the entries of $\XX(t_2)$. We extract from $\XX(t_2)$ the $p + 1$ consecutive indices $i_{-\lfloor p/2\rfloor}, \dots, i_{\lceil p/2\rceil}$ such that $i_0=t_1$. We denote
\begin{gather*}
\ext_e(T)\sei\bigl[i_{-\lfloor p/2\rfloor},i_0\bigr]\times \{t_2\},\quad \ext_f(T)\sei\bigl]i_0,i_{\lceil p/2\rceil}\bigr]\times\{t_2\},\\ \ext(T)\sei\ext_f(T)\cup\ext_e(T),
\end{gather*}
where $\ext_e(T)$ is denoted \emph{edge-extension}, $\ext_f(T)$ is denoted \emph{face-extension} and $\ext(T)$ is just the \emph{extension} of the T-junction $T$.
\end{df}
\begin{df}[{Analysis-Suitability \cite[Definition~2.5]{AS-Tsplines-of-arb-degree}}]
A mesh is \emph{analysis-suitable} if horizontal T-junction extensions do not intersect vertical T-junction extensions.
\end{df}
The main result of this section is the following theorem.
\begin{thm}\label{thm: AS}%
All admissible meshes (in the sense of Definition~\ref{df: admissible mesh}) are analysis-suitable.
\end{thm}
\begin{proof}
We prove the theorem by induction over admissible bisections. We know that the initial mesh $\G_0$ is analysis-suitable because it is a tensor-product mesh without any T-junctions. Consider a sequence $\G_0,\dots,\G_J$ of successive admissible bisections such that $\G_0,\dots,\G_{J-1}$ are analysis-suitable. 
Without loss of generality we shall assume that elements are refined in ascending order with respect to their level, i.e., for $\G_{j+1}=\bisect(\G_j,K_j)$, we assume that $0=\ell(K_0)\le\dots\le\ell(K_{J-1})$. There is such a sequence for any admissible mesh; see the proof of Proposition~\ref{prp: overlays are admissible}.
We have to show that $\G_J$ is analysis-suitable as well.

We denote $K\sei K_{J-1}=[\mu,\mu+\tilde\mu]\times[\nu,\nu+\tilde\nu]\in\G_{J-1}$, and we assume without loss of generality that $\ell(K)$ is even.
The assumption that elements are refined in ascending order with respect to their level implies that no element finer than $K$ has been bisected yet, i.e., 
\begin{equation}\label{eq: ASproof: max ell G}
\max\ell(\G_J)=\ell(K)+1.
\end{equation}
Denote by 
\begin{equation}\label{eq: df Guni}
\Guni k\sei\left\{K'\in\tcup\A\mid\ell(K')=k\right\}\ \in\A
\end{equation}
the $k$-th uniform refinement of $\G_0$. Then $\Guni{\ell(K)+1}$ is a refinement of $\G_J$, in particular 
\begin{equation}\label{eq: ASproof: 1}
\hsk(\G_J)\ \subseteq\ \hsk(\Guni{\ell(K)+1})\ =\ \hsk(\Guni{\ell(K)}),
\end{equation}
since $\ell(K)$ is even. Since $\G_J$ is admissible, all elements in $\pq{\G_J} K$ are at least of level $\ell(K)$ and hence
\begin{equation}\label{eq: ASproof: 2}
\hsk(\G_J)\cap\U K\ \supseteq\ \hsk(\Guni{\ell(K)})\cap\U K.
\end{equation}
and 
\begin{equation}\label{eq: ASproof: 2.5}
\forall\ \tilde K\in\pq{\G_J}K:\quad\size(\ell(\tilde K))\leq\size(\ell(K)) 
\end{equation}
with the level-dependent size
\begin{equation}\label{eq: df size}
\size(\ell(K))\sei(\tilde\mu, \tilde\nu) =\begin{cases}(2^{-\ell(K)/2},2^{-\ell(K)/2})&\text{if $\ell(K)$ even,}\\(2^{-(\ell(K)+1)/2},2^{-(\ell(K)-1)/2})&\text{if $\ell(K)$ odd.}\end{cases}
\end{equation}
Together, \eqref{eq: ASproof: 1} and \eqref{eq: ASproof: 2} read
\begin{equation}\label{eq: ASproof: 3}
\hsk(\G_J)\cap\U K\ =\ \hsk(\Guni{\ell(K)})\cap\U K.
\end{equation}
Consider a T-junction $T\in\mathcal T_J\setminus\mathcal T_{J-1}$ that is generated by the bisection of $K$. Then $T$ is a vertical T-junction on the boundary of $K$, and with \eqref{eq: ASproof: 2.5} follows 
\[\ext(T)\subseteq\bigl\{\mu+\tilde\mu/2\bigr\}\times\bigl[\nu-2^{-\ell(K)/2}\ceilfrac q2,\ \nu+\tilde\nu+2^{-\ell(K)/2}\ceilfrac q2\bigr].\]
Consider an arbitrary horizontal T-junction $\tilde T=(t_1,t_2)\in\mathcal T$. 
We will prove that $\ext(\tilde T)$ does not intersect $\ext(T)$.
From \eqref{eq: ASproof: 1} we conclude that 
$\ext(\tilde T)\subseteq\hsk(\Guni{\ell(K)})$, and \eqref{eq: ASproof: 3} implies that the vertex $\tilde T$ is not in the interior of the $(p,q)$-patch of $K$ and not on its top or bottom boundary, i.e.
\[\tilde T\notin\bigl]\mu-2^{-\ell(K)/2}\floorfrac p2,\ \mu+\tilde\mu+2^{-\ell(K)/2}\floorfrac p2\bigr[\ \times\ \bigl[\nu-2^{-\ell(K)/2}\ceilfrac q2,\ \nu+\tilde\nu+2^{-\ell(K)/2}\ceilfrac q2\bigr].\]
See Figure~\ref{fig: patch without side boundaries} for a sketch.
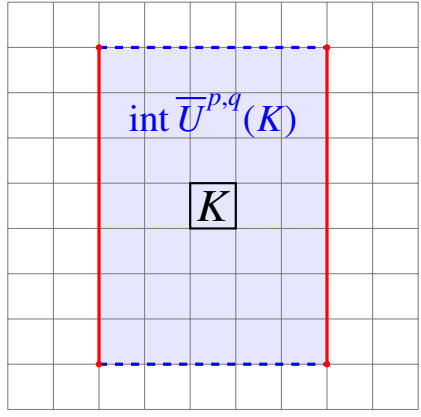
\begin{figure}[t]
\centering
\begin{tikzpicture}[scale=.6]
  \fill[blue!10] (-2,-3) rectangle (3,4);
  \draw[gray] (-4,-4) grid (5,5);
  \node at (.5,.5) {\Large $K$};
  \draw[thick] (0,0) rectangle (1,1);
  \draw[very thick,dashed,blue] (-2,4) coordinate (OL) -- (3,4) coordinate (OR)
                               (-2,-3) coordinate (UL) -- (3,-3) coordinate (UR);
  \draw[very thick,red] (OL)--(UL)  (OR)--(UR);
  \foreach \a in {OR,OL,UR,UL} \fill[red] (\a) circle (2pt);
  \node[blue] at (.5,2.5) {\large $\operatorname{int}\U K$};
\end{tikzpicture}
\caption{Example of the $(p,q)$-patch in a uniform mesh for $p=q=5$. The horizontal T-junction $\tilde T$ may be on a solid red line or outside of $\U K$, but not in the interior of $\U K$ (shaded area) or on the dashed blue lines, which are open at their endpoints.}
\label{fig: patch without side boundaries}
\end{figure}
Assume without loss of generality that $\tilde T$ is on the left side of $K$, this is,
\begin{equation}\label{eq: ASproof: 4}
t_1\leq\mu-2^{-\ell(K)/2}\floorfrac p2. 
\end{equation}
If $\operatorname{type}(\tilde T)={\vdash}$, then the edge-extension $\ext_e(\tilde T)$ points towards $K$ in the sense that 
\begin{align*}
\forall\ (x,t_2)\in\ext(\tilde T)&:\enspace x-t_1\le2^{-\ell(K)/2}\floorfrac p2\stackrel{\eqref{eq: ASproof: 4}}\le \mu-t_1\\
\Leftrightarrow\quad\forall\ (x,t_2)\in\ext(\tilde T)&:\enspace x\le\mu<\mu+\tilde\mu/2.
\end{align*}
This means that $\ext(\tilde T)$ does not intersect $\ext(T)$. See Figure~\ref{fig: extensions do not intersect 1} for an illustration.
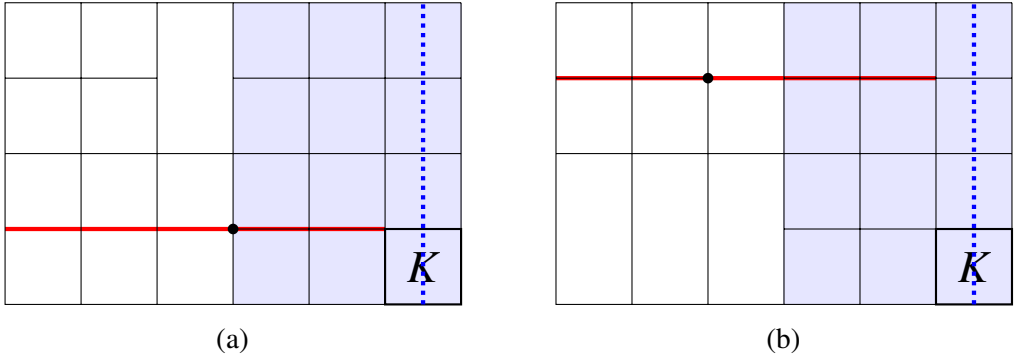
\begin{figure}[t]
\centering
\begin{subfigure}[b]{.4\textwidth}
\centering
\begin{tikzpicture}
\fill[blue!10] (0,0) rectangle (3,4);
\draw[ultra thick, red] (-3,1)--(2,1);
\draw (0,0) grid (3,4);
\draw (0,0)--(-3,0)--(-3,4)--(0,4) (-3,2)--(0,2) (-2,0)--(-2,4) (-1,0)--(-1,4) (-3,3)--(-1,3);
\node at (2.5,.5) {\Large $K$};
\draw[thick] (2,0) rectangle (3,1);
\draw[blue, ultra thick, dotted] (2.5,0)--(2.5,4);
\fill (0,1) circle (2pt);
\end{tikzpicture}
\caption{}\label{fig: extensions do not intersect 1}
\end{subfigure}
\qquad
\begin{subfigure}[b]{.4\textwidth}
\centering
\begin{tikzpicture}
\fill[blue!10] (0,0) rectangle (3,4);
\draw[ultra thick, red] (-3,3)--(2,3);
\draw (0,0) grid (3,4);
\draw (0,0)--(-3,0)--(-3,4)--(0,4) (-3,2)--(0,2) (-2,0)--(-2,4) (-1,0)--(-1,4) (-3,3)--(-1,3);
\node at (2.5,.5) {\Large $K$};
\draw[thick] (2,0) rectangle (3,1);
\draw[blue, ultra thick, dotted] (2.5,0)--(2.5,4);
\fill (-1,3) circle (2pt);
\end{tikzpicture}
\caption{}\label{fig: extensions do not intersect 2}
\end{subfigure}
\caption{In both cases, the T-junction extension $\ext(\tilde T)$ (thick red line) does not intersect the set
$\bigl\{\mu+\tilde\mu/2\bigr\}\times\bigl[\nu-2^{-\ell(K)/2}\ceilfrac q2,\ \nu+\tilde\nu+2^{-\ell(K)/2}\ceilfrac q2\bigr]$
(dotted blue line), which includes $\ext(T)$. The patch $\pq{\G_J}K$ is shaded in light blue.}
\end{figure}

If $\operatorname{type}(\tilde T)={\dashv}$, then there is an odd-level element $K'$ on the right side of $\tilde T$, and two finer even-level elements on the left side. Since there are no elements in $\G_J$ with a level higher than $\ell(K)+1$, which is odd, the two elements on the left side of $\tilde T$ have at most level $\ell(K)$, and hence $\ell(K')\le\ell(K)-1$. Consequently, $K'\notin\pq{\G_J}K$, and the length of the intersection of the face extension $\ext_f(\tilde T)$ with the $(p,q)$-patch of $K$ is at most $2^{-\ell(K)/2}\bigl(\ceilfrac p2-1\bigr)\le2^{-\ell(K)/2}\floorfrac p2$. This leads to the same result as the previous case and is illustrated in Figure~\ref{fig: extensions do not intersect 2}.
Since $\tilde T$ was chosen arbitrary, $\G_J$ is analysis-suitable. This concludes the proof.
\end{proof}

\begin{crl}
All admissible meshes provide T-spline blending functions that are non-negative, linearly indepent, and form a partition of unity \cite{AS-Tsplines-of-arb-degree,isogeom-methods}.
Moreover, on each element $K\in\G\in\A$, there are not more than $2(p+1)(q+1)$ T-spline basis functions that have support on $K$ \cite[Proposition~7.6]{isogeom-methods}.
\end{crl}
This means that on each element, each T-Spline function communicates only with a finite number of other T-spline functions, independent of the total number of functions. This is an important requirement for sparsity of the linear system to be solved in Finite Element Analysis, in the sense that every row and every column of a corresponding stiffness or mass matrix is a sparse vector.

\section{Overlay}\label{sec: overlay}
This section discusses the coarsest common refinement of two meshes 
$\G_1, \G_2 \in \A$, called \emph{overlay}  and denoted by $\G_1 \otimes \G_2$. 
We prove that the overlay of two admissible meshes is also admissible and has bounded cardinality in terms of the involved meshes. This is a classical result in the context of adaptive simplicial meshes and will be crucial for further analysis of adaptive algorithms (cf.\ Assumption (2.10) in \cite{Axioms-of-Adaptivity}).

\begin{df}[Overlay]
We define the operator $\minset$ which yields all minimal elements of a set that is partially ordered by ``$\subseteq$'', \[\minset(\M) \sei  \bigl\{K\in \M \mid \forall K' \in \M: K' \subseteq K \Rightarrow K' = K\bigr\}.\]
The \emph{overlay} of $\G_1,\G_2 \in \A$ is defined by
\begin{align*}
\G_1 \otimes \G_2 \sei  \minset\bigl(\, \G_1 \cup \G_2 \,\bigr).
\end{align*}
\end{df}

\begin{prp}
$\G_1\otimes\G_2$ is the coarsest refinement 
of $\G_1$ and $\G_2$ in the sense that for any $\hat\G$ being a
refinement of $\G_1$ and $\G_2$, and $\G_1\otimes\G_2$ being a refinement 
of $\hat\G$, it follows that $\hat\G=\G_1\otimes\G_2$.
\end{prp}

\begin{proof}
$\G_1$ is a refinement of $\G_2$ if and only if for each $K_1\in\G_1$, there is $K_2\in\G_2$ with $K_1\subseteq K_2$, which is equivalent to $\G_1=\G_1\otimes\G_2$. Given that $\G_1\otimes\hat\G = \hat\G = \G_2\otimes\hat\G$ and $\G_1\otimes\G_2 = (\G_1\otimes\G_2)\otimes\hat\G$, we have
\begin{align*}
\G_1\otimes\G_2
&= (\G_1\otimes\G_2)\otimes\hat\G
= \minset(\G_1\otimes\G_2\cup\hat\G)
\\&= \minset(\minset(\G_1\cup\G_2)\cup\hat\G)
= \minset(\G_1\cup\G_2\cup\hat\G)
\\&= \minset(\G_1\cup\minset(\G_2\cup\hat\G))
= \minset(\G_1\cup\G_2\otimes\hat\G)
\\&= \minset(\G_1\cup\hat\G)
= \G_1\otimes\hat\G
= \hat\G.
\end{align*}\raiseqed
\end{proof}

\begin{prp}\label{prp: overlays are admissible} For any admissible meshes $\G_1,\G_2 \in \A$, the overlay $\G_1\otimes\G_2$ is also admissible.
\end{prp}

\begin{proof}
Consider the set of admissible elements which are coarser than elements of the overlay,
\[\M \sei \bigl\{ K \in {\textstyle\bigcup} \A \mid \exists K' \in \G_1\otimes\G_2: K' \subsetneqq K\bigr\}.\]
Then $\G_1\otimes\G_2$ is the coarsest partition of $\overline\Omega$ into elements from $\tcup\A$ that refines all elements occuring in $\M$.
Note also that $\M$ satisfies
\begin{equation}\label{eq: M is a filter}
\forall\ K,K'\in\tcup\A:\enspace K\in\M\wedge K\subseteq K'\Rightarrow K'\in\M.
\end{equation}
For $j=0,\dots,J=\max\ell(\M )$ and $\bar\G_0\sei\G_0$, set 
\begin{align}
\M_j&\coloneqq\{K\in\M \mid \ell(K)=j\}\notag \\
\text{and}\quad\bar\G_{j+1}&\coloneqq\bisect(\bar\G_j,\M_j).\label{eq: overlay-1}
\end{align}

\noindent\emph{Claim 1. For all $j\in\{0,\dots,J\}$ holds $\M_j\subseteq\bar\G_j$.}
This is shown by induction over $j$. For $j=0$, the claim is true because all admissible elements with zero level are in $\G_0$. Assume the claim to be true for $0,\dots,j-1$ and assume for contradiction that there exists $K\in\M_j\setminus\bar\G_j$. 

Since $K$ has not been bisected yet, $\bar\G_j$ does not contain any $K'$ with $K'\subset K$. Consequently, there exists $K'\in\bar\G_j$ with $K\subset K'$ and hence $\ell(K')<\ell(K)=j$. 
From \eqref{eq: M is a filter} follows $K'\in\M_{\ell(K')}\in\M$, and $\ell(K')<j$ implies that $K'$ has been refined in a previous step. This yields $K'\notin\bar\G_j$, which is the desired contradiction.\bigskip

\noindent\emph{Claim 2. For all $j\in\{0,\dots,J\}$, the bisection \eqref{eq: overlay-1} is admissible.}
Consider $K\in\M_j$ for an arbitrary $j$. By definition of $\M$, there exists $K'\in\G_1\otimes\G_2\subseteq\G_1\cup\G_2$ with $K'\subsetneqq K$. Without loss of generality, we assume $K'\in\G_1$. Since $\G_1\in\A$, there is a sequence of admissible meshes $\G_0=\G_{1\mid 0},\G_{1\mid 1},\dots,\G_{1\mid \mathcal I}=\G_1$ and $i\in\{0,\dots,\mathcal I-1\}$ such that $\G_{1\mid i+1}=\bisect(\G_{1\mid i},\{K\})$. The fact that $\G_{1\mid i+1}\in\A$ (and that levels do not decrease during refinement) implies 
\begin{equation}\label{eq: overlay: pqG1 is fine}
\min\ell(\pq{\G_1}K)\ge\min\ell(\pq{\G_{1\mid i}}K)\ge\ell(K)=j.
\end{equation}
Assume for contradiction that there is $\tilde K\in\pq{\G_j}K$ with $\ell(\tilde K)<\ell(K)=j$. This implies $\tilde K\notin\M$ (otherwise $\tilde K$ would have been bisected in a previous step). Moreover, \eqref{eq: overlay: pqG1 is fine} and Corollary~\ref{crl: old magic patch} yield that there is $\tilde K'\in\pq{\G_1}K$ with $\tilde K'\subset \tilde K$ and hence $\tilde K\in\M$ in contradiction to $\tilde K\notin\M$ from before. This proves Claim~2.

The proven claims show $\M_j=\bar\G_j\setminus\bar\G_{j+1}$ for all $j=0,\dots,J$ and hence 
for the admissible mesh $\bar\G_{J+1}$ that there is no coarser partition of $\overline\Omega$ into elements from $\tcup\A$ that refines all elements in $\M $. This property defines a unique partition and hence 
\[\G_1\otimes\G_2=\bar\G_{J+1}\in\A .\]
\raiseqed
\end{proof}

\begin{lma}\label{lma: bounded overlay}
For all $\G_1,\G_2\in\A$ holds \[\#\left(\G_1\otimes\G_2\right)+\#\G_0\le\#\G_1+\#\G_2\ .\]
\end{lma}
\begin{proof}
By definition, the overlay is a subset of the union of the two involved meshes, i.e., 
\begin{equation}\label{eq: overlay subset of union}
\G_1\otimes\G_2\ =\ \minset(\G_1\cup \G_2)\ \subseteq\ \G_1\cup\G_2\ .
\end{equation}
Define the shorthand notation $\G(K)\sei\{K'\in\G\mid K'\subseteq K\}$.
To prove the lemma, it suffices to show \[\forall\ K\in\G_0,\quad \#(\G_1\otimes\G_2)(K)+1\le \#\G_1(K)+\#\G_2(K)\ .\]
\emph{Case 1.} $\G_1(K)\subseteq(\G_1\otimes\G_2)(K)$. This implies equality and hence
\[\#(\G_1\otimes\G_2)(K)+1 = \#\G_1(K)+1 \le \#\G_1(K)+\#\G_2(K)\ .\]
\emph{Case 2.} There exists $K'\in\G_1(K)\setminus(\G_1\otimes\G_2)(K)$. Then $(\G_1\otimes\G_2)(K)=(\G_1\otimes\G_2)(K)\setminus\{K'\}$ and hence
\begin{align*}
\#(\G_1\otimes\G_2)(K)\ &=\ \#\left((\G_1\otimes\G_2)(K)\setminus\{K'\}\right)
\ \stackrel{\eqref{eq: overlay subset of union}}\le\ \#\left((\G_1\cup\G_2)(K)\setminus\{K'\}\right)\\[.3em]
\ &\le\ \#(\G_1\setminus\{K\})+\#\G_2(K)
\ =\ \#\G_1(K)-1+\#\G_2(K).
\end{align*}
\end{proof}

\section{Nestedness}\label{sec: nestedness}
This section investigates the nesting behavior of the T-spline spaces corresponding to admissible meshes.
In order to prove that nested admissible meshes induce nested spline spaces, we make use of Theorem~6.1 from \cite{ASTS-characterization}. Before presenting the Theorem, we briefly introduce necessary notations.

\begin{df}[Refinement relation]
For any partitions $\G_1,\G_2$ of $\barOmega$, we introduce the refinement relation ``$\preceq$'', which is defined using the overlay (see Section~\ref{sec: overlay}),
\[\G_1\preceq\G_2 \enspace\Leftrightarrow\enspace \G_1\otimes\G_2=\G_2.\]
\end{df}

\begin{crl}\label{crl: refine skeleton}%
Denote the \emph{skeleton} of a mesh $\G$ by $\sk(\G)\sei\hsk(\G)\cup\vsk(\G)$. Then for rectangular partitions $\G_1,\G_2$ of $\barOmega$ holds the equivalence
\[\G_1\preceq\G_2 \enspace\Leftrightarrow\enspace \sk(\G_1)\subseteq\sk(\G_2).\]
\end{crl}

\begin{df}[extended mesh]
Given a rectangular partition $\G$ of $\barOmega$, denote by $\ext(\G)$ the union of all T-junction extensions in the mesh $\G$. 
Then the \emph{extended mesh} $\G^{\ext[]}$ is defined as the unique rectangular partition of $\barOmega$ such that
\[\sk(\G^{\ext[]}) = \sk(\G)\cup\ext(\G).\]
\end{df}

\begin{df}[mesh perturbation]
Given a partition $\G$ of $\barOmega$ into axis-aligned rectangles, we define by $\ptb(\G)$ the set of all continuous and invertible mappings $\delta: \barOmega\to\barOmega$ such that the corners $(0,0)$, $(M,0)$, $(M,N)$, $(0,N)$ are fixed points of $\delta$ and
\[\delta(\G)=\bigl\{\delta(K)\mid K\in\G\bigr\}\]
is also a partition of $\barOmega$ into axis-aligned rectangles. 
\end{df}
This definition differs from the definition of pertubations given in \cite{ASTS-characterization}, which we found difficult to reproduce in a formal manner. The subsequent Proposition~\ref{prp: worst perturbation} shows that our definition includes the understanding of perturbations from \cite{ASTS-characterization}.

\begin{rem}
For $\delta\in\ptb(\G)$, the perturbed mesh $\delta(\G)$ has the skeleton $\sk(\delta(\G))=\delta(\sk(\G))$. Hence, global index vectors can be defined according to Definition~\ref{df: global indices}, and since all T-junctions in $\delta(\G)$ are of axis-parallel types ($\vdash,\bot,\dashv,$ or $\top$), we can also apply Definition~\ref{df: TJ-extensions} for T-junction extensions in the perturbed mesh. Note in particular that the perturbation $\delta$ does not in general map T-junction extensions to the corresponding extensions in the perturbed mesh, i.e., if $T$ is a T-junction in $\G$, then 
\[\ext_{\delta(\G)}(\delta(T))\ \neq\ \delta(\ext_{\G}(T)).\]
\end{rem}

\begin{prp}\label{prp: worst perturbation}%
For any rectangular partition $\G$ of $\barOmega$, there is some $\delta^*\in\ptb(\G)$ such that any two T-junction face extensions in $\delta^*(\G)$ are disjoint.
\end{prp}
In the context of \cite{ASTS-characterization}, this means that $\delta^*(\G)$ has no crossing vertices and no overlap vertices.

\begin{proof}
If all T-junction extensions in $\G$ are pairwise disjoint, then $\delta^*$ is the identity map. If there exist T-junctions $T_1,T_2$ in $\G$ with intersecting face extensions, then $T_1$ and $T_2$ are either both vertical or both horizontal T-junctions.
Assume w.l.o.g.\ that $T_1$ and $T_2$ are vertical T-junctions. Since their (vertical) face extensions overlap, both T-junctions have the same $x$-coordinate $t_0$. Let $T_1=(t_0,t_1)$ and $T_2=(t_0,t_2)$, and assume $t_1<t_2$. There exists $t_{1.5}$ with $t_1\le t_{1.5}\le t_2$ such that at least one of the open segments $\{t_0\}\times(t_1,t_{1.5})$ and $\{t_0\}\times(t_{1.5},t_2)$ does not intersect with the vertical skeleton $\vsk(\G)$. Assume that $\{t_0\}\times(t_1,t_{1.5})\cap\vsk(\G)=\emptyset$ and define 
\begin{align*}
\barOmega_{x=t_0}&\sei\left\{(x,y)\in\barOmega\mid x=t_0\right\}\\
\text{and}\enspace\G_{x=t_0}&\sei\left\{K\in\G\mid K\cap \barOmega_{x=t_0}\neq\emptyset\right\}.
\end{align*}
Let $h$ be the length of the shortest edge in $\G$, and set $\varepsilon\sei h/2$. 
We define $\delta_{T_1T_2}$ by
\[\delta_{T_1T_2}(x,y) = 
\begin{cases}
(x,y) &\text{if }(x,y)\in \bigcup(\G\setminus\G_{x=t_0})\\
(x-\varepsilon,y) & \text{if }x=t_0 \text{ and } y<t_1 \\
(x+\varepsilon,y) & \text{if }x=t_0 \text{ and } y>t_{1.5}\\
\bigl(x+\tfrac{\varepsilon\,(2y-t_1-t_{1.5})}{t_{1.5}-t_1},y\bigr) & \text{if }x=t_0 \text{ and } t_1\le y\le t_{1.5}
\end{cases}
\]
and elsewhere by horizontal linear interpolation, which is illustrated in Figure~\ref{fig: perturbation example}.
The map $\delta_{T_1T_2}$ then satisfies the following properties.
\begin{enumerate}
\item $\delta_{T_1T_2}$ is in $\ptb(\G)$.
\item The T-junction extensions of $\delta_{T_1T_2}(T_1)$ and $\delta_{T_1T_2}(T_2)$ do not intersect.
\item $\delta_{T_1T_2}$ does not lead to intersecting of T-junction extensions that did not intersect in the unperturbed mesh $\G$.
\end{enumerate}
\begin{figure}[ht]
\centering
\begin{subfigure}[b]{.4\textwidth}
\centering
\begin{tikzpicture}[scale=1.25]
\coordinate (0) at (1.5,0);
\coordinate (1) at (1.5,1);
\coordinate (2) at (1.5,2);
\coordinate (3) at (1.5,3);
\foreach \a in {.25,.5,.75,1}
{
  \draw[red!50] ($\a*(0) + (1,0)-\a*(1,0)$)--($\a*(1) + (1,1)-\a*(1,1)$)--($\a*(2) + (1,2)-\a*(1,2)$)--($\a*(3) + (1,3)-\a*(1,3)$);
  \draw[red!50] ($\a*(0) + (2,0)-\a*(2,0)$)--($\a*(1) + (2,1)-\a*(2,1)$)--($\a*(2) + (2,2)-\a*(2,2)$)--($\a*(3) + (2,3)-\a*(2,3)$);
}
\foreach \a in {1,...,24}
 \draw[red!50] (1,.125*\a)--(2,.125*\a);
\fill[gray!40] (0,0) rectangle (1,3) (2,0) rectangle (3,3);
\draw[thick] (0,0) grid (3,3)  (.5,1)--(.5,2) (0)--(1) (2)--(3) (2.5,0)--(2.5,2);
\end{tikzpicture}
\caption{The unperturbed mesh $\G$.}
\end{subfigure}
\begin{subfigure}[b]{.4\textwidth}
\centering
\begin{tikzpicture}[scale=1.25]
\coordinate (0) at (1.25,0);
\coordinate (1) at (1.25,1);
\coordinate (2) at (1.75,2);
\coordinate (3) at (1.75,3);
\foreach \a in {.2,.4,.6,.8,1}
{
  \draw[red!50] ($\a*(0) + (1,0)-\a*(1,0)$)--($\a*(1) + (1,1)-\a*(1,1)$)--($\a*(2) + (1,2)-\a*(1,2)$)--($\a*(3) + (1,3)-\a*(1,3)$);
  \draw[red!50] ($\a*(0) + (2,0)-\a*(2,0)$)--($\a*(1) + (2,1)-\a*(2,1)$)--($\a*(2) + (2,2)-\a*(2,2)$)--($\a*(3) + (2,3)-\a*(2,3)$);
}
\foreach \a in {1,...,30}
 \draw[red!50] (1,.1*\a)--(2,.1*\a);
\fill[gray!40] (0,0) rectangle (1,3) (2,0) rectangle (3,3);
\draw[thick] (0,0) grid (3,3)  (.5,1)--(.5,2) (0)--(1) (2)--(3) (2.5,0)--(2.5,2);
\end{tikzpicture}
\caption{The perturbed mesh $\delta_{T_1T_2}(\G)$.}
\end{subfigure}
\caption{Example for a perturbation $\delta_{T_1T_2}$. In the shaded area, $\delta_{T_1T_2}$ equals the identity map. In the non-shaded region, we underlaid a red grid to illustrate the behavior of $\delta_{T_1T_2}$.}
\label{fig: perturbation example}
\end{figure}
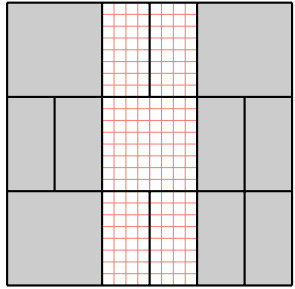
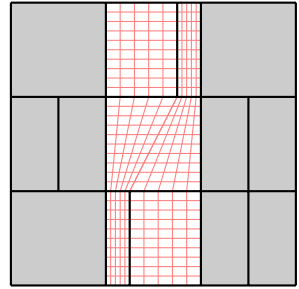
A straight-forward proof shows that perturbations can be concatenated in the sense that 
\[\delta_1\in\ptb(\G),\ \delta_2\in\ptb(\delta_1(\G))\enspace\Rightarrow\enspace
\delta_2\circ\delta_1\in\ptb(\G). \]
This allows for the subsequent conclusion of the proof.
Given the mesh $\G_0\sei\G$ choose an arbitrary pair $(T_0,T_0')$ of T-junctions in $\G$ such that their face extensions intersect, and set $\G_1\sei\delta_{T_0T_0'}(\G_0)$. Then choose $(T_1,T_1')$ such that $T_1$ and $T_1'$ are T-junctions with intersecting face extensions in $\G_1$, construct $\delta_{T_1T_1'}$ as above, accounting that $h$ and $\varepsilon$ may have changed. Set $\G_2\sei\delta_{T_1T_1'}(\G_1)$. Repeat this until in a mesh $\G_n$, there are no intersecting T-junction face extensions. Then $\delta^*\sei\delta_{T_{n-1}T_{n-1}'}\circ\dots\circ\delta_{T_0T_0'}$
is in $\ptb(\G)$ and satisfies that all T-junction face extensions in $\delta^*(\G)$ are pairwise disjoint.
\end{proof}

\begin{thm}[{\cite[Theorem~6.1]{ASTS-characterization}}]\label{thm: ASTS-characterization}%
Given two analysis-suitable meshes $\G_1$ and $\G_2$, if for all $\delta\in\ptb(\G_2)$ holds 
\[\bigl(\delta(\G_1)\bigr)^{\ext[]} \preceq \bigl(\delta(\G_2)\bigr)^{\ext[]},\]
then the T-spline spaces corresponding to $\G_1$ and $\G_2$ are nested.
\end{thm}

The main result of this section is the following.

\begin{thm}\label{thm: nesting}%
Any two meshes $\G_1,\G_2\in\A$ that are nested in the sense $\G_1\preceq\G_2$ satisfy for all $\delta\in\ptb(\G_2)$
\[\bigl(\delta(\G_1)\bigr)^{\ext[]} \preceq \bigl(\delta(\G_2)\bigr)^{\ext[]}.\]
\end{thm}

\begin{proof}
According to Corollary~\ref{crl: refine skeleton}, we have to show that
\[\ext\bigl(\delta(\G_1)\bigr)\cup\sk\bigl(\delta(\G_1)\bigr)\ \subseteq\ \ext\bigl(\delta(\G_2)\bigr)\cup\sk\bigl(\delta(\G_2)\bigr).\]
We prove this for $\G_2$ being an admissible bisection of $\G_1$. The claim then follows inductively for all admissible  refinements of $\G_1$.
Let $K\in\G_1\in\A$ and $\G_2\sei\bisect(\G_1,K)\in\A$. Since ``$\preceq$'' denotes an elementwise subset relation, it is preserved under the mapping $\delta$. Thus, from $\G_1\preceq\G_2$ follows $\delta(\G_1)\preceq\delta(\G_2)$ and consequently  $\sk\bigl(\delta(\G_1)\bigr) \subseteq \sk\bigl(\delta(\G_2)\bigr)$.
It remains to prove that
\[\ext\bigl(\delta(\G_1)\bigr)\ \subseteq\ \ext\bigl(\delta(\G_2)\bigr)\cup\sk\bigl(\delta(\G_2)\bigr).\]
Denote by $\mathcal T_1$ and $\mathcal T_2$ the set of T-junctions in $\G_1$ and $\G_2$, respectively. 
Assume w.l.o.g.\ that $\ell(K)$ is even, and consider an arbitrary T-junction $T^\delta$ in the mesh $\delta(\G_1)$.
Since $\delta$ is continuous and invertible, there is a one-to-one correspondence between the T-junctions in $\G_1$ and $\delta(\G_1)$, i.e., there is $T\in\mathcal T_1$ with $\delta(T)=T^\delta$, and $T$ and $T^\delta$ are of the same type ($\vdash,\bot,\dashv,$ or $\top$).

\emph{Case 1.} $T\notin K$. Then $T$ is still a T-junction after bisecting $K$, i.e., $T\in\mathcal T_2$. Consequently, $T^\delta$ is also a T-junction in $\delta(\G_2)$.

\emph{Case 1a.} $T$ is a vertical T-junction. Since $\ell(K)$ is assumed to be even, its bisection does not affect the horizontal skeleton, i.e., $\hsk(\G_1)=\hsk(\G_2)$ and hence $\hsk(\delta(\G_1))=\hsk(\delta(\G_2))$. Consquently, the T-junction extensions of $T$ and $T^\delta$ are preserved, 
\[\ext_{\G_1}(T)=\ext_{\G_2}(T)\enspace\text{and}\enspace\ext_{\delta(\G_1)}(T^\delta)=\ext_{\delta(\G_2)}(T^\delta)\subseteq\ext\big(\delta(\G_2)\big).\]

\emph{Case 1b.} $T$ is a horizontal T-junction. We will show that the corresponding T-junction extension in the pertubed mesh is preserved, i.e., \[\ext_{\delta(\G_1)}(T^\delta)=\ext_{\delta(\G_2)}(T^\delta).\]
Assume for contradiction that $\ext_{\delta(\G_1)}(T^\delta)\neq\ext_{\delta(\G_2)}(T^\delta)$. The bisection of $K$ generates a vertical edge $E_K\supseteq\vsk(\G_2)\setminus\vsk(\G_1)$, and we denote 
\[E_K^\delta\sei\delta(E_K)\supseteq\vsk(\delta(\G_2))\setminus\vsk(\delta(\G_1)).\]
Obviously, $E_K^\delta$ intersects with $\ext_{\delta(\G_1)}(T^\delta)$, otherwise the T-junction extension would be the same in $\delta(\G_2)$. Given $K=[\mu,\mu+\tilde\mu]\times[\nu,\nu+\tilde\nu]$, we define the half-open domain $K\ho\sei\left]\mu,\mu+\tilde \mu\right[\times[\nu,\nu+\tilde\nu]$, which is the rectangle $K$ without its vertical edges. Then $E_K\subset K\ho$ and hence $E_K^\delta\subset K\ho^\delta\sei\delta(K\ho)$. Together, we have that $\ext_{\delta(\G_1)}(T^\delta)$ intersects with $K\ho^\delta$.
Since the bisection of $K$ is admissible, we know from the proof of Theorem~\ref{thm: AS} that $\ext_{\G_1}(T)$ does not intersect with $K\ho$ in the unperturbed mesh $\G_1$. Define the $T$-environment \[\U[]T\sei\mbigcup[\substack{K'\in\G_1\\K'\ho\cap\ext(T)\neq\emptyset}]K',\]
as the union of all $K'\in\G_1$ such that $\ext(T)$ intersects the corresponding half-open $K'\ho$.
Then $\U[]T$ is a rectangular domain that does not intersect with $K\ho$. Since for each $K'\subseteq\U[]T$, the image $\delta(K')$ is a rectangle and since $\delta$ is continuous, $\delta\big(\U[]T\big)$ is a rectangular domain that does not intersect with $K^\delta\ho$.
Moreover, since all edges and vertices in $\U[]T$ are continuously mapped into $\delta\big(\U[]T\big)$, we have 
$\U[]{T^\delta}\subseteq\delta\big(\U[]T\big)$.
Together, we get that $\U[]{T^\delta}$ does not intersect with $K^\delta\ho$, hence $\ext_{\delta(\G_1)}(T^\delta)$ does not intersect with $K^\delta\ho$, which is the desired contradiction.

\emph{Case 2.} $T\in K$. In Section~\ref{sec: refinement}, we assumed that $p,q\ge2$. This implies that all neighbors of $K$ are in $\pq{\G_1}K$ and that $K$ is in the patch of all those neighbors as well. Since $\G_1$ is admissible, the level of a neighbor of $K$ is either $\ell(K)$ or $\ell(K)+1$. Since $\ell(K)$ is even, $T$ must be a vertical T-junction, and $T^\delta$ is a vertical T-junction as well. Since $T$ is on the boundary of $K$, and the bisection of $K$ generates a vertical edge, $T$ is not a T-junction anymore in $\G_2$. Hence $T^\delta$ is a vertex, but not a T-junction in $\delta(\G_2)$. The T-junction extension $\ext(T^\delta)$ hence only exists in $\delta(\G_1)$. Consider the edge extension of $T^\delta$.

\emph{Case 2a.} $\ext_e(T^\delta)\subseteq\vsk(\delta(\G_2))$. There is no problem with that.

\emph{Case 2b.} $\ext_e(T^\delta)\nsubseteq\vsk(\delta(\G_2))$. Then there exists some $\tilde T^\delta\in\ext_e(T^\delta)$ which is a T-junction in $\delta(\G_2)$, such that
\[\ext_e(T^\delta)\subset\ext_{\delta(\G_2)}(\tilde T^\delta)\subseteq\ext(\delta(\G_2)).\]

The Cases 2a and 2b hold analogously for the face extension $\ext_f(T^\delta)$. Together, we have
\[\ext(T^\delta)\subseteq\ext(\delta(\G_2))\cup\vsk(\delta(\G_2)),\]
which concludes the proof.
\end{proof}

The combination of Theorem \ref{thm: ASTS-characterization} and \ref{thm: nesting} reads as follows.
\begin{crl}
For any two meshes $\G_1,\G_2\in\A$ that are nested in the sense $\G_1\preceq\G_2$, the corresponding T-spline spaces are also nested.
\end{crl}

\section{Linear Complexity}\label{sec: complexity}
This section is devoted to a complexity estimate in the style of a famous estimate for the  Newest Vertex Bisection on triangular meshes given by Binev, Dahmen and DeVore \cite{BDV} and, in an alternative version, by Stevenson~\cite{Stevenson}. The estimate reads as follows.

\begin{thm}\label{thm: complexity}
Any sequence of admissible meshes $\G_0,\G_1,\dots,\G_J$ with \[\G_j=\refine(\G_{j-1},\M_{j-1}),\quad\M_{j-1}\subseteq\G_{j-1}\quad\text{for}\enspace j\in\{1,\dots,J\}\] satisfies
\[\left|\G_J\setminus\G_0\right|\ \le\ C_{p,q}\sum_{j=0}^{J-1}|\M_j|\ ,\]
with $C_{p,q}=(3+\sqr2)(4d_p+1)(4d_q+\sqr2)$ and $d_p,d_q$ from Lemma~\ref{lma: K1 in refMS => K2 in S} below.
\end{thm}
\begin{rem}
Theorem~\ref{thm: complexity} shows that, with regard to possible mesh gradings, the refinement algorithm is as flexible as successive bisection without the  closure step. 
However, this result is non-trivial. Given a mesh $\G\in\A$ and an element $K\in\G$ to be bisected, there is no uniform bound on the number of generated elements $\#(\refine(\G,\{K\})\setminus\G)$. This is illustrated by the following example.
\end{rem}
\begin{ex}\label{ex: local complexity}
Consider the case $p=q=2$ and the initial mesh $\G_0$ given through $M=3$ and $N=4$. Mark the element in the lower left corner of the mesh and compute the corresponding refinement $\G_1$; repeat this step $k$ times. Then there exists an element $K_k$ in $\G_k$ such that $\#(\refine[1,1](\G_k,K_k)\setminus\G_k)\ge k$. This is illustrated in Figure~\ref{fig: bad complexity example}.
\begin{figure}[ht]
\centering
\includegraphics[width=.3\textwidth]{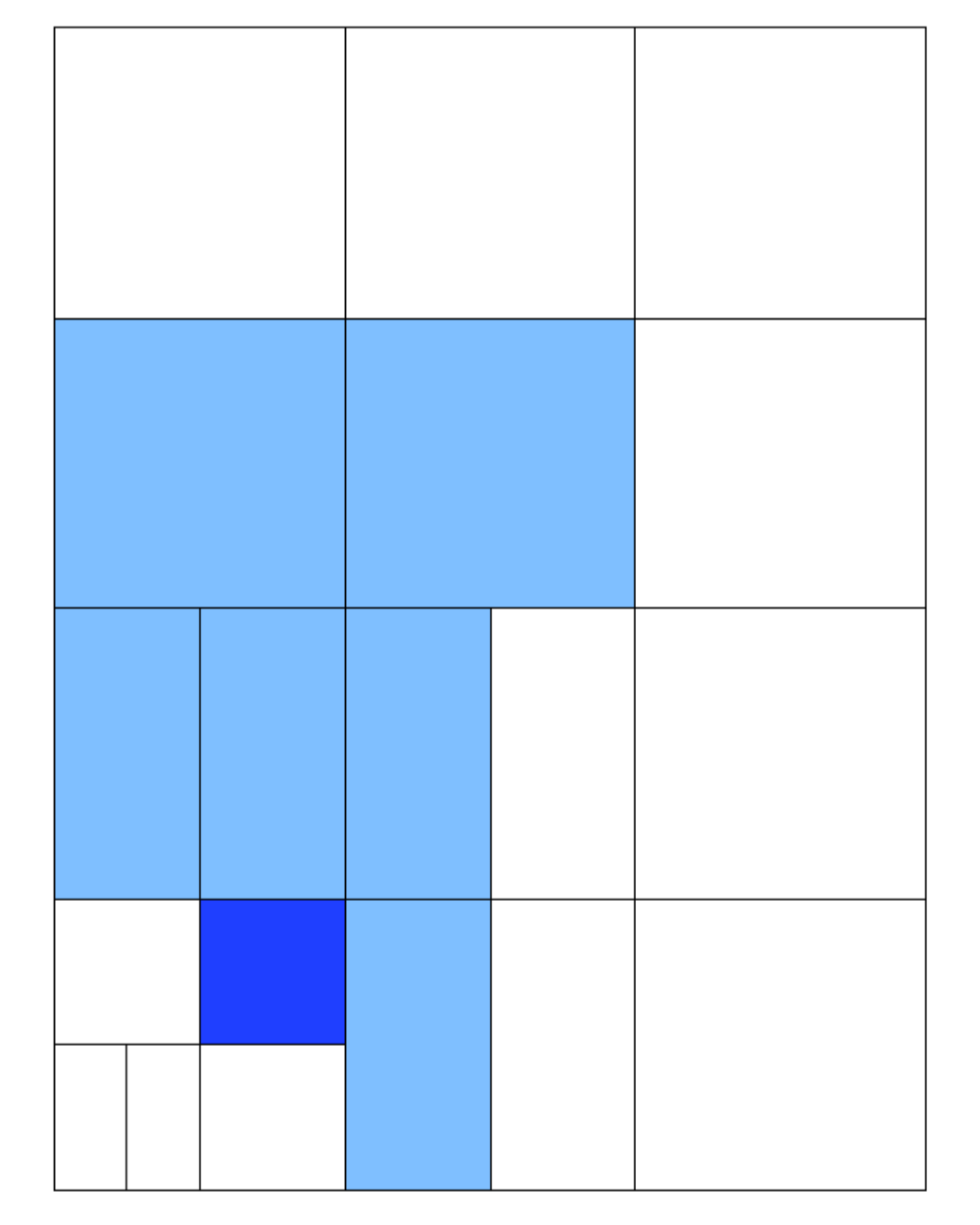}
\hspace{.1\textwidth}
\includegraphics[width=.3\textwidth]{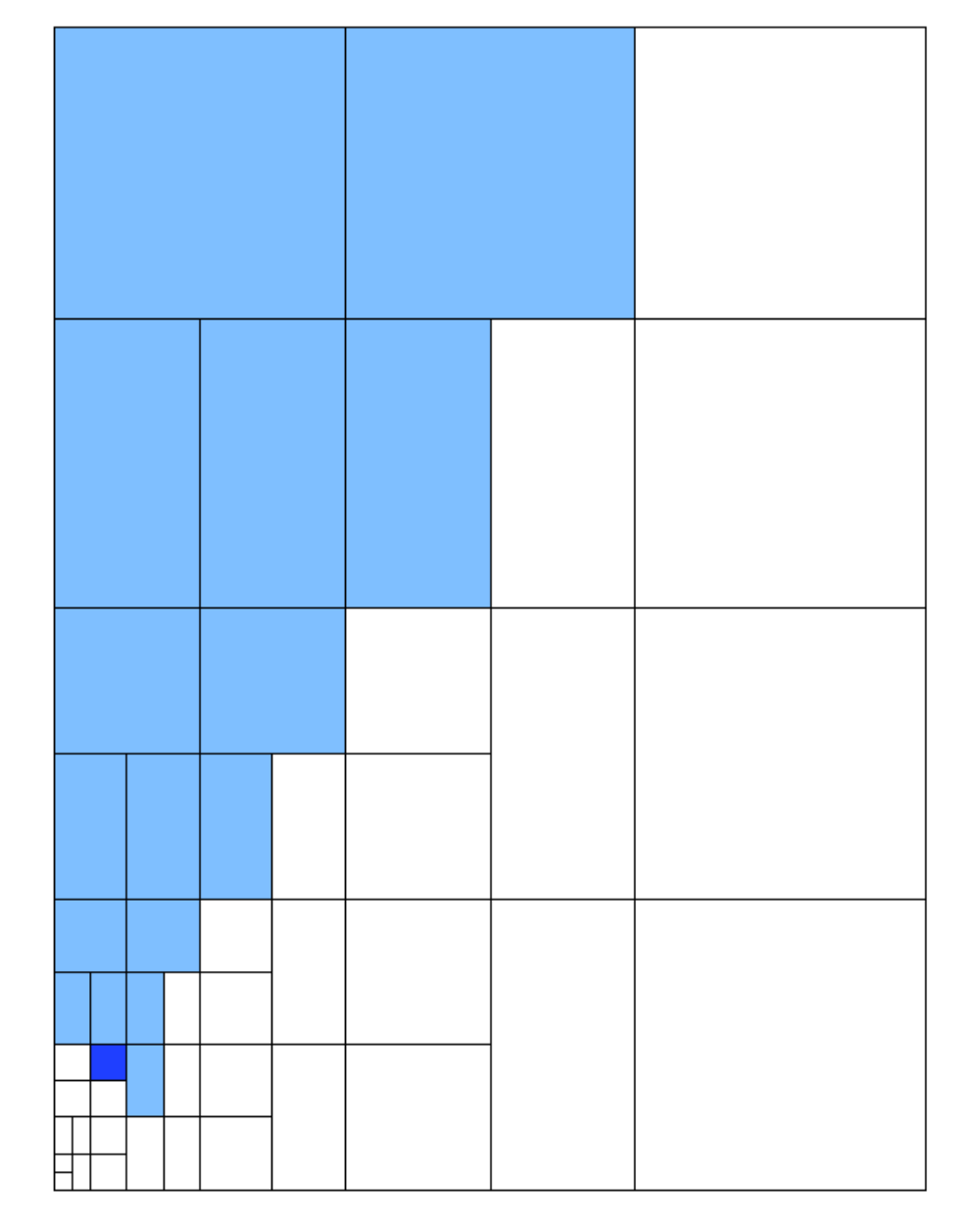}
\caption{The mesh $\G_3$ and the mesh $\G_8$ from Example~\ref{ex: local complexity}. The rectangles $K_3$ and $K_8$ are marked blue. The closures $\clos[1,1](\G_3,\{K_3\})$ and $\clos[1,1](\G_8,\{K_8\})$ are marked in light blue. Since the closure of $K_3$ consists of 7 elements, 14 elements will be generated if $K_3$ is bisected. Analogously, marking $K_8$ would cause the generation of 34 new elements.}
\label{fig: bad complexity example}
\end{figure}
\end{ex}

\begin{ex}
The large constant $C_{p,q}$ is not observed in practise. For $p=q=3$, we constructed for each $J\in\{1,\dots,2000\}$ a sequence $\G_0,\G_1,\dots,\G_J$ with $G_{j+1}=\bisect(G_j,K_j)$ and $K_j\in\G_j$ of uniform random choice. The ratio $\left|\G_J\right|/J$ was below $6$ (see Figure~\ref{fig: nice complexity example}), instead of the theoretical upper bound $C_{3,3}\approx12\,996$ from Theorem~\ref{thm: complexity}. We applied this procedure for $p,q=2,\dots,9$. The results are listed in Figure~\ref{fig: maxfac}.
In Figure~\ref{fig: medfac}, we listed similar results for $J\in\{1,\dots,100\}$, always marking the element in the lower left corner. In that case, the observed ratios are higher, but still orders of magnitude below the corresponding theoretical bounds.
\begin{figure}[ht]
\centering
\includegraphics[width=\textwidth]{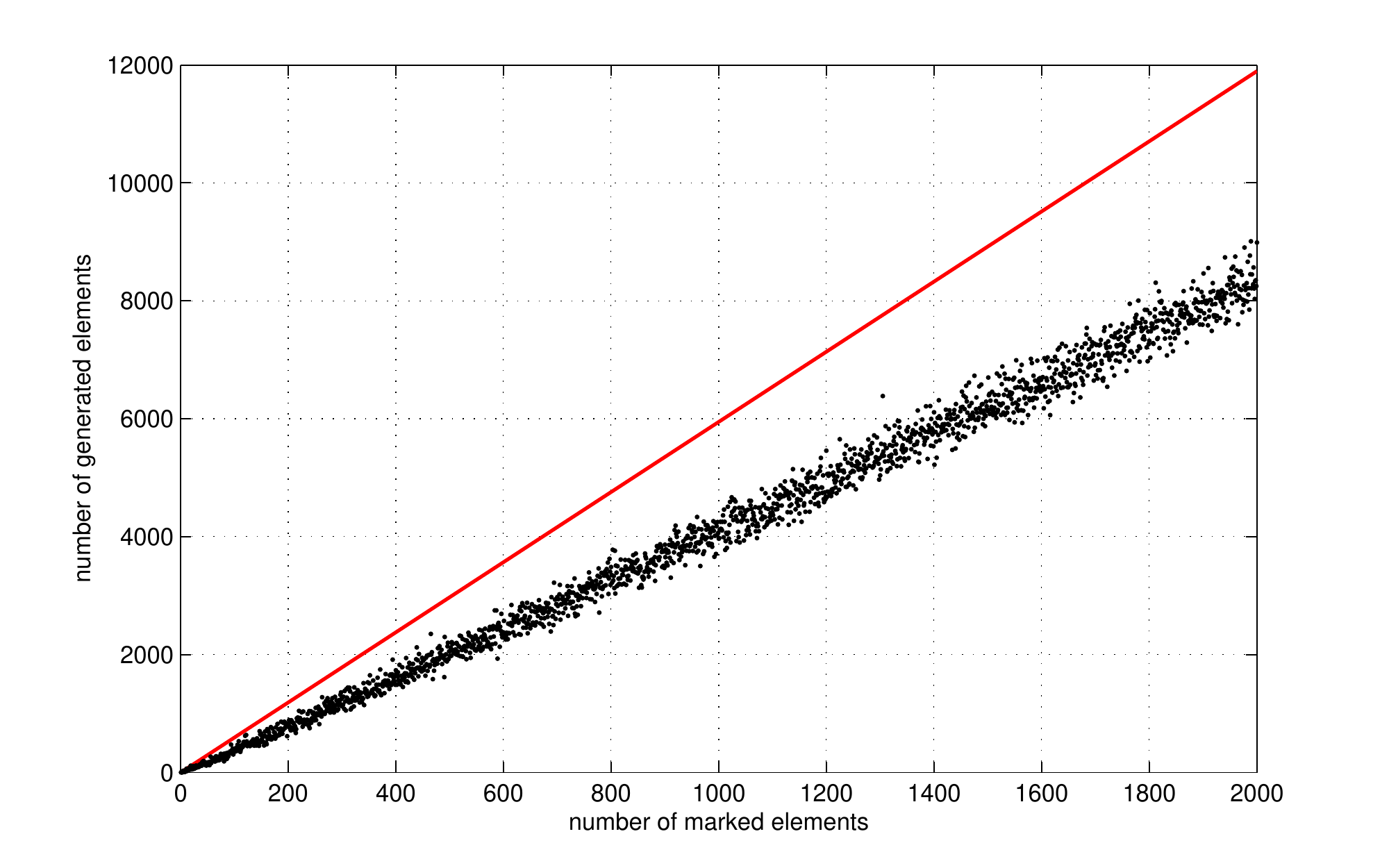}
\caption{Generated and marked elements for randomly refined $(3,3)$-admissible meshes. Each black dot corresponds to a sequence of random admissible refinements. The red line depicts the highest observed ratio ($\approx 5.95$). The median of the observed ratios is $\approx4.09$.}
\label{fig: nice complexity example}
\end{figure}
\end{ex}
\begin{figure}[ht]
\centering
\begin{tabular}{c|*{8}c}
\backslashbox{$p$}{$q$} & 2 & 3 & 4 & 5 & 6 & 7 & 8 & 9 \\\hline
2 &    5 &    5 &    7 &    7 &    7 &    7 &    8 &    8 \\ 
3 &    6 &    6 &    7 &    7 &    8 &    8 &    9 &   11 \\ 
4 &    7 &    8 &    8 &    8 &   11 &   10 &   10 &   12 \\ 
5 &    7 &    7 &    9 &   10 &   10 &   12 &   11 &   13 \\ 
6 &    7 &    8 &   10 &   10 &   11 &   12 &   12 &   16 \\ 
7 &    8 &   11 &   10 &   13 &   12 &   12 &   16 &   14 \\ 
8 &    9 &   10 &   11 &   17 &   13 &   13 &   15 &   15 \\ 
9 &    9 &   11 &   12 &   14 &   14 &   16 &   16 &   23 \\ 
\end{tabular}
\caption{Maximal observed ratios of generated and marked elements for random refinement.}
\label{fig: maxfac}
\end{figure}
\begin{figure}[ht]
\centering
\begin{tabular}{c|*{8}c}
\backslashbox{$p$}{$q$} & 2 & 3 & 4 & 5 & 6 & 7 & 8 & 9\\\hline
2&       24&       33&       46&       56&       69&       78&       91&      100 \\ 
3&       33&       46&       65&       78&       97&      109&      128&      140 \\ 
4&       46&       65&       91&      110&      136&      154&      179&      198 \\ 
5&       56&       78&      110&      132&      163&      186&      216&      238 \\ 
6&       69&       97&      136&      164&      202&      229&      268&      295 \\ 
7&       78&      110&      154&      186&      229&      260&      304&      335 \\ 
8&       91&      128&      180&      217&      268&      304&      355&      391 \\ 
9&      100&      141&      198&      239&      295&      335&      391&      431 \\ 
\end{tabular}
\caption{Maximal observed ratios of generated and marked elements when refining the lower left corner.}
\label{fig: medfac}
\end{figure}
We devote the rest of this section to proving Theorem~\ref{thm: complexity}.
\begin{lma}\label{lma: K1 in refMS => K2 in S}
Given $\M\subseteq\G\in\A$ and $K\in\refine(\G,\M)\setminus\G$, there exists $K'\in \M$ such that $\ell(K)\le\ell(K')+1$ and \[\Dist(K,K')\le2^{-\ell(K)/2}(d_p,\,d_q),\]
with ``$\le$'' understood componentwise and constants 
\[d_p\sei\tfrac12+(1+\sqr2)(p+\sqr2),\quad d_q\sei\tfrac1{\sqr2}+(2+\sqr2)(q+\sqr2).\]
\end{lma}
\begin{proof}
The coefficient $\D(k)$ from Definition~\ref{df: magic patch} is bounded by 
\[\D(k)\le\bigl((p+\sqr2)\,2^{-1-k/2},\ (q+\sqr2)\,2^{-(k+1)/2}\bigr)\quad\text{for all }k\in\mathbb N.\]
Hence for $\tilde K\in\G\in\A$, any $\tilde K'\in\pq\G{\tilde K}$ satisfies 
\begin{equation}\label{eq: magic radius}
\Dist(\tilde K,\tilde K')\le2^{-\ell(\tilde K)/2}\,\Bigl(\tfrac{p+\sqr2}2,\enspace\tfrac q{\sqr2}+1\Bigr).
\end{equation}
The existence of $K\in\refine(\G,\M)\setminus\G$ means that Algorithm~\ref{alg: refinement} bisects $K'=K_J,K_{J-1},\dots,K_0$ such that $K_{j-1}\in\pq\G{K_j}$ and $\ell(K_{j-1})<\ell(K_j)$ for $j=J,\dots,1$, having $K'\in \M$ and $K\in\child(K_0)$, with `$\child$' from Definition~\ref{df: bisection}. Lemma~\ref{lma: levels change slowly} yields $\ell(K_{j-1})=\ell(K_j)-1$ for $j=J,\dots,1$, which allows for the estimate
\begin{align*}
\Dist(K',K_0)\enspace&\le\enspace \sum_{j=1}^J\Dist(K_j,K_{j-1})
\enspace \stackrel{\eqref{eq: magic radius}}\le \enspace \sum_{j=1}^J2^{-\ell(K_j)/2}\,\Bigl(\tfrac{p+\sqr2}2,\enspace\tfrac q{\sqr2}+1\Bigr) \\
&= \sum_{j=1}^J2^{-(\ell(K_0)+j)/2}\,\Bigl(\tfrac{p+\sqr2}2,\enspace\tfrac q{\sqr2}+1\Bigr) \\
&< 2^{-\ell(K_0)/2}\,\Bigl(\tfrac{p+\sqr2}2,\enspace\tfrac q{\sqr2}+1\Bigr)\sum_{j=1}^\infty2^{-j/2} \\
&= (1+\sqr2)\,2^{-\ell(K_0)/2}\,\Bigl(\tfrac{p+\sqr2}2,\enspace\tfrac q{\sqr2}+1\Bigr) \\
&= (2+2\sqr2)\,2^{-\ell(K)/2}\,\Bigl(\tfrac{p+\sqr2}2,\enspace\tfrac q{\sqr2}+1\Bigr).
\end{align*}
The estimate $\Dist(K_0,K)\le 2^{-2-\ell(K_0)/2}\,\bigl(1,\sqr2\bigr)$ and a triangle inequality conclude the proof.
\end{proof}

\begin{proof}[Proof of Theorem~\ref{thm: complexity}]\ 

\pnumpx For $K\in\tcup\A$ and $\KM\in\M\sei\M_0\cup\dots\cup\M_{J-1}$, define $\lambda(K,\KM)$ by \[\lambda(K,\KM)\sei\begin{cases}2^{(\ell(K)-\ell(\KM))/2}&\text{if }\ell(K)\le\ell(\KM)+1\text{ and }\Dist(K,\KM)\le2^{1-\ell(K)/2}(d_p,d_q),\\[.3em]0&\text{otherwise.}\end{cases}\]

\pnumpx[Main idea of the proof.]
\begin{alignat*}{2}
\left|\G_J\setminus\G_0\right| &= \mathbox[2.5em]{\sum_{K\in\G_J\setminus\G_0}}1 &&\Stackrel{\numref{sum_lambda > 1}}\le\sum_{K\in\G_J\setminus\G_0}\sum_{\KM\in\M}\lambda(K,\KM) \\
&\Stackrel{\numref{summe aller lambdas beschraenkt}}\le\sum_{\KM\in\M} C_{p,q} &&=C_{p,q}\,\sum_{j=0}^{J-1} |\M_j|.
\end{alignat*}

\pnumpx[For all $j\in\{0,\dots,J-1\}$ and $\KM\in\M_j$ holds \[\sum_{K\in\G_J\setminus\G_0}\lambda(K,\KM)\ \le\ (3+\sqr2)(4d_p+1)(4d_q+\sqr2)\ =\ C_{p,q}\ .\]]%
\label{summe aller lambdas beschraenkt}%
This is shown as follows. By definition of $\lambda$, we have
\begin{align*}
\mathbox[1cm]{\sum_{K\in\G_J\setminus\G_0}}\lambda(K,\KM)
&\le \mathbox[1cm]{\sum_{K\in\bigcup\A\setminus\G_0}}\lambda(K,\KM)\\
&= \sum_{j=1}^{\ell(\KM)+1}2^{(j-\ell(\KM))/2}\,\#\underbrace{\bigl\{K\in\tcup\A\mid\ell(K)=j\text{ and }\Dist(K,\KM)\le 2^{1-j/2}(d_p,d_q)\bigl\}}_B.
\end{align*}
Since we know by definition of the level that $\ell(K)=j$ implies $|K|=2^{-j}$, we know that $2^j\left|\tcup B\right|$ is an upper bound of $\#B$. The rectangular set $\tcup B$ is the union of all admissible elements of level $j$ having their midpoints inside an rectangle of size \[2^{2-j/2}d_p\,\times\,2^{2-j/2}d_q.\] An admissible element of level $j$ is not bigger than $2^{-j/2}\,\times\,2^{(1-j)/2}$. Together, we have \[\left|\tcup B\right|\le 2^{-j}(4d_p+1)(4d_q+\sqr2),\] and hence $\#B\le (4d_p+1)(4d_q+\sqr2)$. The claim is shown with 
\[\sum_{j=1}^{\ell(\KM)+1}2^{(j-\ell(\KM))/2}=\sum_{j=1-\ell(\KM)}^12^{j/2}<\sqr2+\sum_{j=0}^\infty2^{-j/2}=\tfrac{2\sqr2-1}{\sqr2-1}=3+\sqr2.\]

\pnumpx[Each $K\in\G_J\setminus\G_0$ satisfies \[\sum_{\KM\in\M}\lambda(K,\KM)\ \ge\ 1.\]]%
\label{sum_lambda > 1}%
Consider $K\in\G_J\setminus\G_0$. Set $j_1<J$ such that $K\in\G_{j_1+1}\setminus\G_{j_1}$. Lemma~\ref{lma: K1 in refMS => K2 in S} states the existence of $K_1\in\M_{j_1}$ with $\Dist(K,K_1)\le2^{-\ell(K)/2}(d_p,d_q)$ and $\ell(K)\le\ell(K_1)+1$. Hence $\lambda(K,K_1)=2^{\ell(K)-\ell(K_1)}>0$.
The repeated use of Lemma~\ref{lma: K1 in refMS => K2 in S} yields $j_1>j_2>j_3>\dots$ and $K_2,K_3,\dots$ with $K_{i-1}\in\G_{j_i+1}\setminus\G_{j_i}$ and $K_i\in\M_{j_i}$ such that 
\begin{equation}\label{eq: complexity -last}
\Dist(K_{i-1},K_i)\le2^{-\ell(K_{i-1})/2}(d_p,d_q)\enspace\text{and}\enspace\ell(K_{i-1})\le\ell(K_i)+1.
\end{equation}
We repeat applying Lemma~\ref{lma: K1 in refMS => K2 in S} as $\lambda(K,K_i)>0$ and $\ell(K_i)>0$, and we stop at the first index $L$ with $\lambda(K,K_L)=0$ or $\ell(K_L)=0$. 
If $\ell(K_L)=0$ and $\lambda(K,K_L)>0$, then
\[\sum_{\KM\in\M}\lambda(K,\KM)\ge\lambda(K,K_L)=2^{(\ell(K)-\ell(K_L))/2}\ge\sqr2.\]
If $\lambda(K,K_L)=0$ because $\ell(K)>\ell(K_L)+1$, then \eqref{eq: complexity -last} yields $\ell(K_{L-1})\le\ell(K_L)+1<\ell(K)$ and hence
\[\sum_{\KM\in\M}\lambda(K,\KM)\ge\lambda(K,K_{L-1})=2^{(\ell(K)-\ell(K_{L-1}))/2}\ge\sqr2.\]
If $\lambda(K,K_L)=0$ because $\Dist(K,K_L)>2^{1-\ell(K)/2}(d_p,d_q)$, then a triangle inequality shows
\[
2^{1-\ell(K)/2}(d_p,d_q)\ <\ \Dist(K,K_1)+\sum_{i=1}^{L-1}\Dist(K_i,K_{i+1}) 
\ \le\ 2^{-\ell(K)/2}(d_p,d_q)+\sum_{i=1}^{L-1} 2^{-\ell(K_i)/2}(d_p,d_q),
\]
and hence $\smash{\displaystyle 2^{-\ell(K)/2} \le\sum_{i=1}^{L-1} 2^{-\ell(K_i)/2}}$. The proof is concluded with 
\[ 1\ \le\ \sum_{i=1}^{L-1} 2^{(\ell(K)-\ell(K_i))/2}\ =\ \sum_{i=1}^{L-1} \lambda(K,K_i)\ \le\ \sum_{\KM\in\M}\lambda(K,\KM).\]\raiseqed
\end{proof}

\section{Conclusion}\label{sec: conclusions}
We presented an adaptive refinement algorithm for a subclass of analysis-suitable T-meshes that produces nested T-spline spaces, and we proved theoretical properties that are crucial for the analysis of adaptive schemes driven by a posteriori error estimators. As an example, compare the assumptions (2.9) and (2.10) in \cite{Axioms-of-Adaptivity} to Theorem~\ref{thm: complexity} and Lemma~\ref{lma: bounded overlay}, respectively.
The presented refinement algorithm can be extended to the three-dimensional case, which is our current work. The factor $C_{p,q}$ from the complexity estimate is affine in each of the parameters $p,q$ and increases exponentially with growing dimension.
We aim to apply the proposed algorithm to proof the rate-optimality of an adaptive algorithm for the numerical solution of second-order linear elliptic problems using T-splines as ansatz functions. Similar results have been proven for simple FE discretizations of the Poisson model problem in 2007 by Stevenson \cite{Stevenson}, in 2008 by Cascon, Kreuzer, Nochetto and Siebert \cite{CKNS}, and recently for a wide range of discretizations and model problems by Carstensen, Feischl, Page and Praetorius \cite{Axioms-of-Adaptivity}. 

\section*{Acknowledgements}
The authors gratefully acknowledge support by the Deutsche
Forschungsgemeinschaft in the Priority Program 1748 ``Reliable
simulation techniques in solid mechanics. Development of non-standard 
discretization methods, mechanical and mathematical
analysis'' under the project ``Adaptive isogeometric modeling of 
propagating strong discontinuities in heterogeneous materials''.


\providecommand{\bysame}{\leavevmode\hbox to3em{\hrulefill}\thinspace}
\providecommand{\MR}{\relax\ifhmode\unskip\space\fi MR }
\providecommand{\MRhref}[2]{%
  \href{http://www.ams.org/mathscinet-getitem?mr=#1}{#2}
}
\providecommand{\href}[2]{#2}

\end{document}